\documentclass[11pt]{amsart}
\usepackage{xcolor,cancel}
\usepackage{color}
\usepackage{graphicx}
\usepackage[normalem]{ulem}
\usepackage[latin1]{inputenc}
\usepackage{pb-diagram}
\usepackage{mathrsfs}

\usepackage{amsmath,amsthm,amscd}
\usepackage[psamsfonts]{amssymb}
\usepackage[all]{xy}
\usepackage{wrapfig}
\usepackage{bbm}
\usepackage{here}

\newtheorem{Thm}{Theorem}[section]
\newtheorem{Prop}[Thm]{Proposition}
\newtheorem{Lem}[Thm]{Lemma}
\newtheorem{Cor}[Thm]{Corollary}

\newtheorem{Property}[Thm]{Property}

\theoremstyle{remark}
\newtheorem{Rem}[Thm]{Remark}

\theoremstyle{definition}
\newtheorem{Def}[Thm]{Definition}

\def\Z{{\mathbb Z}}
\def\R{{\mathbb R}}
\def\Q{{\mathbb Q}}

\def\calA{\mathscr{A}}

\def\calG{\mathscr{G}}

\def\calM{\mathscr{M}}

\def\calR{\mathscr{R}}

\def\Diff{\mathrm{Diff}}

\def\fEmb{\mathrm{Emb}^{\mathrm{fr}}}

\def\ve{\varepsilon}

\def\even{\mathrm{even}}
\def\odd{\mathrm{odd}}
\newcommand{\p}{\partial}
\def\Cr{{\mathrm Cr}}

\def\psc{\mathsf{psc}}

\begin{document}

\title[Families of diffeomorphisms and concordances]{Families of
  diffeomorphisms and concordances \\ detected by trivalent graphs}

\author{Boris Botvinnik}
\address{
Department of Mathematics\\
University of Oregon \\
Eugene, OR, 97405\\
USA
}
\thanks{BB was partially supported by Simons collaboration grant 708183}
\email{botvinn@uoregon.edu}

\author{Tadayuki Watanabe} \address{Department of Mathematics \\ Kyoto
  University \\ Kyoto 606-8502 \\ Japan }
\email{tadayuki.watanabe@math.kyoto-u.ac.jp}
\thanks{TW was partially
  supported by JSPS Grant-in-Aid for Scientific Research 21K03225 and by
  RIMS, Kyoto University}
   \date{\today} \maketitle
\vspace*{-10mm}
\begin{abstract}
 We study families of diffeomorphisms detected by trivalent graphs via
 the Kontsevich classes. We specify some recent results and
 constructions of the second named author to show that those
 non-trivial elements in homotopy groups
 $\pi_*(B\Diff_{\p}(D^d))\otimes \Q$ are lifted to homotopy groups of
 the moduli space of $h$-cobordisms $\pi_*(B\Diff_{\sqcup}(D^d\times
 I))\otimes \Q$. As a geometrical application, we show that those
 elements in $\pi_*(B\Diff_{\p}(D^d))\otimes \Q$ for $d\geq 4$ are
 also lifted to the rational homotopy groups
 $\pi_*(\calM^\psc_{\p}(D^d)_{h_0})\otimes \Q$ of the moduli space of
 positive scalar curvature metrics. Moreover, we show that the same
 elements come from the homotopy groups $\pi_*(\calM^\psc_{\sqcup}
 (D^d\times I; g_0)_{h_0})\otimes \Q$ of moduli space of concordances
 of positive scalar curvature metrics on $D^d$ with fixed round metric
 $h_0$ on the boundary $S^{d-1}$.
\end{abstract}  

\renewcommand{\baselinestretch}{0.6}\normalsize
\tableofcontents

\renewcommand{\baselinestretch}{1.2}\normalsize
\section{Results}\label{s:results}
\subsection{Extension of graph surgery to concordance}
Let $\Diff_\partial(D^d)$ be the group of diffeomorphisms $\phi :
D^d\to D^d$ which restrict to the identity near the boundary $\p
D^d=S^{d-1}$.

Recently, the second author obtained the following theorem.
\begin{Thm}[\cite{Wa09,Wa18a,Wa18b,Wa21}]\label{thm:nontriv}
Let $d\geq 4$. For each $k\geq 2$, the evaluation of Kontsevich's
characteristic classes on $D^d$-bundles gives an epimorphism
\begin{equation*}
 \pi_{k(d-3)}B\Diff_\partial(D^d)\otimes\Q
\to \calA_k^{\even/\odd}
\end{equation*}
to the space of $\calA_k^{\even/\odd}$ of trivalent graphs.  For
$k=1$, the same result holds for the group $
\pi_{2n-2}B\Diff_\partial(D^{2n+1})\otimes\Q $ for many odd integers
$d=2n+1\geq 5$ satisfying some technical condition\footnote{$d=5,7,9,11,15,19,23,24,25,\ldots$, checked by non-integrality of some rational numbers involving the Bernoulli numbers in \cite{Wa09}. Actually, this holds for all $d\geq 5$ odd (\cite{KrRW}). See also Remark~\ref{rem:pi}}.
\end{Thm}
Theorem~\ref{thm:nontriv} was proved by evaluating Kontsevich's
characteristic classes (\cite{Kon}) on elements constructed by surgery
on trivalent graphs embedded in $D^d$.

Here we recall the definition of the spaces $\calA_k^{\even/\odd}$ of
connected trivalent graphs, which are the trivalent parts of
Kontsevich's graph homology \cite{Kon}. In general, trivalent graph
has even number of vertices, and if it is $2k$, then the number of
edges is $3k$. Let $V(\Gamma)$ and $E(\Gamma)$ denote the sets of
vertices and edges of a trivalent graph $\Gamma$,
respectively. Labellings of a trivalent graph $\Gamma$ are given by
bijections $V(\Gamma)\to \{1,2,\ldots,2k\}$, $E(\Gamma)\to
\{1,2,\ldots,3k\}$. Let $\calG_k$ be the vector space over $\Q$
spanned by the set $\calG_k^0$ of all labelled connected trivalent
graphs with $2k$ vertices modulo isomorphisms of labelled graphs. The
version $\calA_k^\even$, which works for even-dimensional manifolds,
is defined by
\begin{equation*}
\calA_k^\even=\calG_k/{\mbox{IHX, label change}},
\end{equation*}
where the IHX relation is given in Figure~\ref{fig:IHX}
\begin{figure}[!htbp]
\begin{center}
\includegraphics[height=1in]{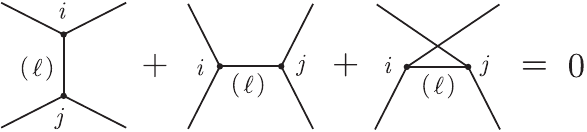}
\end{center}
\caption{IHX relation.}\label{fig:IHX}
\end{figure}

\noindent
and the label change relation is generated by the following relations:
\begin{equation*}
\Gamma'\sim -\Gamma,\qquad \Gamma''\sim \Gamma.
\end{equation*}
Here, $\Gamma'$ is the graph obtained from $\Gamma$ by exchanging
labels of two edges, $\Gamma''$ is the graph obtained from $\Gamma$ by
exchanging labels of two vertices. The version $\calA_k^\odd$, which
works for odd-dimensional manifolds, can be defined similarly as
$\calA_k^\even$ except a small modification in the orientation
convention. Namely, let $\widetilde{\calG}_k$ be the vector space over
$\Q$ spanned by the set $\widetilde{\calG}_k^0$ of all pairs
$(\Gamma,o)$ of labelled connected trivalent graphs $\Gamma$ with $2k$
vertices modulo isomorphisms of labelled graphs and orientations $o$
of the real vector space $H^1(\Gamma;\R)$. Then we define
\begin{equation*}
\calA_k^\odd=\widetilde{\calG}_k/{\mbox{IHX, label change, orientation
    reversal}}
\end{equation*}
where the IHX and the label change
relation is the same as above, and the orientation reversal is the
following:
\begin{equation*}
(\Gamma, -o)\sim -(\Gamma,o). 
\end{equation*}
Let $X$ be a $d$-dimensional path connected smooth manifold with
non-empty boundary.  We also let $\Diff_{\sqcup}(X\times
I):=\Diff(X\times I, X\times\{0\}\cup \partial X\times I)$ be the
group of pseudoisotopies. There is a natural fiber sequence
\begin{equation}\label{eq:diff}
\Diff_{\p}(X\times I) \xrightarrow{i}
\Diff_{\sqcup}({X}\times I) \xrightarrow{\p}
\Diff_{\p}({X}\times \{1\}),
\end{equation}
where $i: \Diff_{\p}({X}\times I) \to \Diff_{\sqcup}({X}\times I) $ is
the inclusion, and $\p: \Diff_{\sqcup}({X}\times I) \to
\Diff_{\p}({X}\times \{1\})$ restricts a diffeomorphism $\psi :
     {X}\times I \to {X}\times I$ to the top part of the boundary
     $\psi|_{{X}\times \{1\}}$. This gives a corresponding fiber
     sequence of the classifying spaces
\begin{equation}\label{eq:bdiff}
B\Diff_{\p}({X}\times I) \xrightarrow{i}
B\Diff_{\sqcup}({X}\times I) \xrightarrow{\p}
B\Diff_{\p}({X}\times \{1\}).
\end{equation}
\begin{Rem} The group of pseudoisotopies $\Diff_{\sqcup}(X\times
    I)$ is often denoted as $C_{\p}(X)$.
\end{Rem}  
The first main result of this paper is the following.
\begin{Thm}[Theorem~\ref{thm:bordism-simplify}]\label{thm:main1}
  Let $d\geq 4$. All the elements given by surgery on
    trivalent graphs with $2k$ vertices, $k \geq 1$, are in the image
    of the homomorphism
\begin{equation*} 
{\p_*: \pi_{k(d-3)}B\Diff_{\sqcup}(X\times
  I) \to \pi_{k(d-3)}B\Diff_{\p}(X).}
\end{equation*}
Let $m := 1$ if $d$ even and $m := (d -1)/2$ for $d$ odd. Furthermore,
each element in the group ${ \pi_{k(d-3)}B\Diff_{\p}(X)}$ constructed
by surgery on a trivalent graph embedded in {$X$} has a lift in the
group ${\pi_{k(d-3)}B\Diff_{\sqcup}(X\times I)}$ represented by a
smooth {$(X\times I)$}-bundle over $S^{k(d-3)}$ that
admits a fibrewise handle decomposition with a single
  handle pair in each fibre, of indices $m$ and $m+1$.
  Moreover, each pair is geometrically cancelling (up to an
  appropriate isotopy).
\end{Thm}
\begin{Rem}
The reader who is familiar with Cerf's graphic
(\cite{Ce,HW,Igu} etc.) would find that this condition is equivalent
to that the family admits a graphic consisting of a single ``lens''.
\end{Rem}  
\begin{Rem}
A version of this theorem for bordism group was pointed out to the
second author by Peter Teichner (\cite[Theorem~9.3]{Wa20}).  We would
like to emphasize the following new features in
Theorem~\ref{thm:main1}:
\begin{enumerate}
\item The crucial feature that will be used in our
  applications to rational homotopy groups of the moduli spaces of
  metrics of positive scalar curvature is that our handlebodies are
  built using a single $m$-handle and a $S^{k(d-3)}$-family of $(m +
  1)$-handles that geometrically cancel it (up to an
  appropriate isotopy).
\item We describe in this paper the details about
  the interpretation of the graph surgery in terms of spherical
  modifications along framed Hopf links, which were sketched in
  \cite[\S{9}]{Wa20}. This could also be applied to constructions of
  families of embeddings in a manifold.
\end{enumerate}
\end{Rem}
\begin{Cor}\label{cor:concordance-nontrivial}
  Let $d\geq 4$ and $k\geq 2$. If $d$ is even (resp. if $d$ is odd),
  then $\pi_{k(d-3)}B\Diff_{\sqcup}(D^d\times I)\otimes\Q$
  is nontrivial whenever $\calA_k^{\even}$ (resp. $\calA_k^\odd$) is
  nontrivial.  For $k=1$, $\pi_{2n-2}B\Diff_{\sqcup}(D^{2n+1}\times
  I)\otimes\Q$ is nontrivial for many odd integers $d=2n+1\geq 5$
  satisfying the same technical condition as in
    Theorem~\ref{thm:nontriv}.
\end{Cor}
\begin{Rem}\label{rem:pi}
\begin{enumerate}
\item Note that this includes results for pseudoisotopies of $D^4$. It
  was proved in \cite{Wa18b} that $\pi_2 B\Diff_\sqcup(D^4\times
  I)\otimes\Q$ is nonzero.  Theorem \ref{thm:main1} shows that
  $\pi_kB\Diff_\sqcup(D^4\times I)\otimes \Q$ are non-trivial for many
  $k>2$.  This is new result.
\item Recently, A. Kupers and O. Randal-Williams (\cite{KuRW}),
  M. Krannich and O. Randal-Williams (\cite{KrRW}) computed the rational
  homotopy groups of $B\Diff_\partial(D^d)$ in some wide range of
  dimensions surprisingly completely. In particular, it follows
  from their results that for $n>5$, the natural map
\begin{equation*}
\pi_{2n-2}B\Diff_\sqcup(D^{2n+1}\times I)\otimes\Q\to
\pi_{2n-2}B\Diff_\partial(D^{2n+1})\otimes\Q
\end{equation*}
is an isomorphism and both terms are isomorphic to $\Q\oplus
(K_{2n-1}(\Z)\otimes\Q)$. In particular,
Corollary~\ref{cor:concordance-nontrivial} for $d=2n+1>11$ and $k=1$
follows from their results.
\end{enumerate}
\end{Rem}

\subsection{Application to the moduli space of psc-metrics}
Let $h_0$ be the standard round metric on $S^{d-1}=\p D^d$, and
$\calR_{\p}(D^d)_{h_0}$ be the space of Riemannian metrics $g$ on the disk
$D^d$ which have a form $h_0+dt^2$ near the boundary $S^{d-1}$. The group
$\Diff_\partial(D^d)$ acts on $\calR_{\p}(D^d)_{h_0}$ by pulling a metric
back: $g\cdot \phi\mapsto \phi^*g$. It is easy to see that this action
is free, and, since the space $\calR_{\p}(D^d)_{h_0}$ is contractible,
there is a homotopy equivalence
\begin{equation*}
B\Diff_\partial(D^d)\sim \calM_{\p}(D^d)_{h_0} :=
\calR_{\p}(D^d)_{h_0}/\Diff_\partial(D^d).
\end{equation*}
Thus the moduli space $\calM_{\p}(D^d)_{h_0}$ could be thought as a
geometrical model of the classifying space $B\Diff_\partial(D^d)$.
Below we identify the spaces $\calM_{\p}(D^d)_{h_0}$ and
$B\Diff_\partial(D^d)$.  Let $\calR^\psc(D^d)_{h_0}\subset
\calR_{\p}(D^d)_{h_0}$ be a subspace of metrics with positive scalar
curvature (which will abbreviated as ``psc-metrics''). We have the
following diagram of principal $\Diff_\partial(D^d)$-fiber bundles:
\begin{equation*}  
\begin{diagram}
   \setlength{\dgARROWLENGTH}{1.95em}
  \node{\calR^\psc_{\p}(D^d)_{h_0}}
          \arrow{e,t}{i}
          \arrow{s,b}{p}
  \node{\calR_{\p}(D^d)_{h_0}}
          \arrow{s,b}{\bar p}
  \\
  \node{\calM^\psc_{\p}(D^d)_{h_0}}
           \arrow{e,t}{\iota}
  \node{\calM_{\p}(D^d)_{h_0}}
\end{diagram}
\end{equation*}
Here $\calM^\psc_{\p}(D^d)_{h_0}:=\calR^\psc_{\p}(D^d)_{h_0}/\Diff_\partial(D^d)$ is
the moduli space of psc-metrics. 
\begin{Thm}\label{thm:M}
Let $d\geq 4$ be an integer. All classes given by surgery on
trivalent graphs are in the image of the induced map
\begin{equation*}
  \iota_*:\pi_q\calM^\psc_{\p}(D^d)_{h_0}\otimes\Q\to
  \pi_qB\Diff_\partial(D^d)\otimes\Q.
\end{equation*}
Hence, all nontrivial elements of $\pi_qB\Diff_\partial(D^d)\otimes\Q$
given by surgery on trivalent graphs lift to nontrivial elements of
$\pi_q\calM^\psc_{\p}(D^d)_{h_0}\otimes\Q$.
\end{Thm}
\begin{Rem}
For $d\geq 6$, Theorem \ref{thm:M} follows also from \cite[Theorem
  F]{ERW}.  We give a geometrical proof of Theorem \ref{thm:C} which,
in particular, proves Theorem \ref{thm:M}. In fact, our proof
  proves a stronger statement for the existence of the lift in the
  moduli space. Namely, all classes given by surgery on trivalent
graphs are in the image of the induced map
$\iota_*:\pi_q\calM_\partial^\psc(X)_{h_0}\to \pi_qB\Diff_\partial(X)$
for an arbitrary smooth manifold $X$ of dimension $d\geq 4$ having a
psc metric $h_0$.  For $d=4$, Theorem \ref{thm:M} is in contrast
  to that the classes in $\pi_qB\Diff(X)$ detected by Seiberg--Witten
  theory do not admit fiberwise psc-metrics (\cite{Ru,Konn}).
\end{Rem}
Next, we fix some geometrical data. Consider the subset $ (D^d\times
\{0\})\cup (S^{d-1}\times I) \subset D^d\times I $ and fix a
psc-metric $g_0\in \calR^{\psc}(D^d\times \{0\})_{h_0}$.  We view the
cylinder $D^d\times I$ as a manifold with corners. Let $U$ be a colar
of $(D^d\times \{0\})\cup (S^{d-1}\times I) $; we assume that $U$ is
parametrized by $(x,t,s)$ near the corner $S^{d-1}\times \{0\}$, as it
it shown in Figure ~\ref{fig:fig01a}, where $x\in S^{d-1}\times
\{0\}$.
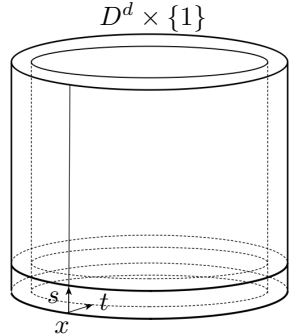
\begin{figure}[!htbp]
\setlength{\unitlength}{1mm}
  \begin{picture}(0,0)
    \put(-7,0){{\small $D^d\times \{1\}$}}
\put(-7,-38){{\small $t$}}
\put(-14,-37){{\small $s$}}
\put(-13,-41){{\small $x$}}
  \end{picture} 
\begin{center}
\includegraphics[height=1.5in]{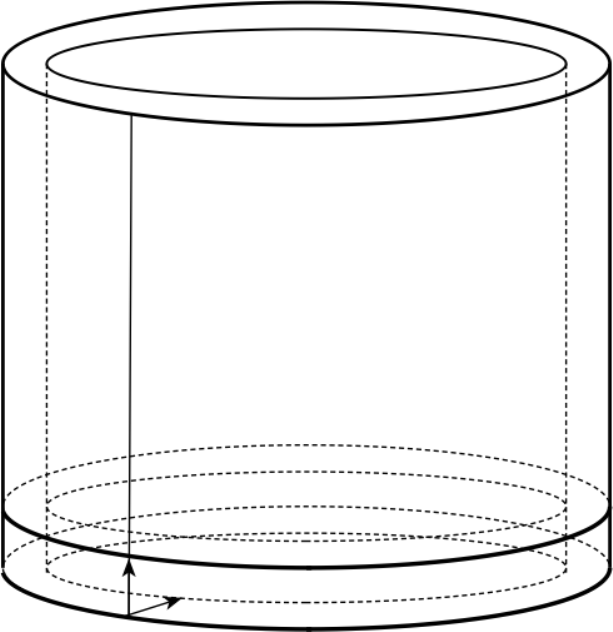}
\end{center}
\caption{A collar of $(D^d\times \{0\})\cup (S^{d-1}\times I) \to
D^d\times I$.}\label{fig:fig01a}
\end{figure}

\noindent
We consider a subspace $\calR_\sqcup(D^d\times I;g_0) \subset
\calR(D^d\times I)$ of Riemannian metrics $\bar g$ which restrict to
\begin{equation}\label{eq:01}
  \left\{
  \begin{array}{ll}
    g_0 + ds^2 & \mbox{near $D^{d}\times \{0\}$} \\ h_0+ds^2+dt^2 &
    \mbox{near $S^{d-1}\times I$} \\ g+ds^2 & \mbox{near $D^{d}\times
      \{1\}$ for some $g\in \calR(D^d\times \{1\})_{h_0}$}
  \end{array}\right.
\end{equation}
Let $\calR^\psc_\sqcup(D^d\times I;g_0)_{h_0}\subset
\calR_\sqcup(D^d\times I;g_0)_{h_0}$ be a corresponding subspace of
psc-metrics. 
Again, we notice that the group $\Diff_{\sqcup}(D^d\times I)$ acts
freely on a contractible space $\calR_\sqcup(D^d\times I;g_0)_{h_0}$.
In particular, we have homotopy equivalence
\begin{equation*}
  B\Diff_\sqcup (D^d\times I) \sim \calM_{\sqcup} (D^d\times I; g_0)_{h_0}
  := \calR_\sqcup(D^d\times I;g_0)_{h_0}/\Diff_\sqcup (D^d\times I) \ .
\end{equation*}
Again we have 
the following diagram of principal
$\Diff_\sqcup (D^d\times I)$-fiber bundles:
\begin{equation*}  
\begin{diagram}
   \setlength{\dgARROWLENGTH}{1.95em}
  \node{\calR^\psc_{\sqcup} (D^d\times I; g_0)_{h_0}}
          \arrow{e,t}{i}
          \arrow{s,b}{p}
  \node{\calR_{\sqcup} (D^d\times I; g_0)_{h_0}}
          \arrow{s,b}{\bar p}
  \\
  \node{\calM^\psc_{\sqcup} (D^d\times I; g_0)_{h_0}}
           \arrow{e,t}{\iota}
  \node{\calM_{\sqcup} (D^d\times I; g_0)_{h_0}}
\end{diagram}
\end{equation*}
We also notice that the restriction map
\begin{equation*}
  \calR^\psc_{\sqcup} (D^d\times I; g_0)_{h_0} \to \calR^\psc_{\p}(D^d)_{h_0},
  \ \ \ \bar g\mapsto
g=\bar g|_{D^d\times \{1\}}
\end{equation*}
where $g$ is given in (\ref{eq:01}), induces a map of corresponding
moduli spaces:
\begin{equation*}
  \partial^\psc:\calM^\psc_{\sqcup} (D^d\times I; g_0)_{h_0}\to
  \calM^\psc_{\p}(D^d)_{h_0}
\end{equation*}
\begin{Thm}\label{thm:C}
Let $d\geq 4$ be an integer.  All lifts in
$\pi_q\calM^\psc_{\p}(D^d)_{h_0}\otimes \Q$ found in Theorem~\ref{thm:M}
are in the image of the homomorphism
\begin{equation*}
  \partial^\psc_* : \pi_{k(d-3)}\calM^\psc_{\sqcup} (D^d\times I; g_0)_{h_0}\otimes\Q
  \to \pi_{k(d-3)} \calM^\psc_{\p}(D^d)_{h_0}\otimes\Q \ .
 \end{equation*} 
Hence, any nontrivial elements of $\pi_qB\Diff_\partial(D^d)\otimes\Q$
given by surgery on trivalent graphs lift to nontrivial elements of
$\pi_{k(d-3)}\calM^\psc_\sqcup(D^d\times I;g_0)_{h_0}\otimes\Q$.
\end{Thm}
\subsection{Conventions}
\begin{itemize}
\item A {\it $(W,\partial W)$-bundle} is a smooth $W$-bundle with
  structure group $\Diff_\partial(W)$. We say that two $(W,\partial
  W)$-bundles $\pi_i:E_i\to B$ ($i=0,1$) are {\it concordant} if the
  disjoint union, the $(W,\partial W)$-bundle $\pi_0\sqcup
  \pi_1:E_0\sqcup E_1\to B\sqcup B$ over disjoint copies of $B$ is the
  restriction of a $(W,\partial W)$-bundle
  $\widetilde{\pi}:\widetilde{E}\to [0,1]\times B$ to
  $\widetilde{\pi}^{-1}(\{0,1\}\times B)$. If
    $W=X\times I$ for a compact manifold $X$, we denote by
    $\partial_{\sqcup}W = X\times \{0\}\cup \p X \times I$. Then a
  {\it $(W,\partial_{\sqcup}W)$-bundle} 
  is a smooth $W$-bundle with
  structure group $\Diff_\sqcup(W)$. Concordance between two
  $(W,\partial_\sqcup W)$-bundles is defined similarly as above.
\item A {\it framed embedding} (or a {\it framed link}) consists of an
  embedding $\varphi:S\to X$ between smooth manifolds and a choice of
  a normal framing $\tau$ on $\varphi(S)$, where by a normal framing
  we mean a trivialization $\nu(\varphi(S))\cong \varphi(S)\times
  \R^{\mathrm{codim}\,\varphi(S)}$ of the normal bundle.
\item We will often say ``a framed embedding $\varphi$'' or ``a framed
  link $\varphi$'', instead of $(\varphi,\tau)$.
\item We consider links as submanifolds equipped with
  parametrizations. Thus in this paper links are embeddings. Also, we
  assume that famlies of links are smoothly parametrized.
\item We will consider \emph{trivialities} of families or bundles in
  several different meanings. Instead of saying just ``trivial
  bundle'', we will say that a bundle/family is \emph{trivialized} if
  it is equipped with a trivialization. If it admits at least one
  trivialization, we say it is \emph{trivializable}. A given family
  $\{\varphi_s\}$ of some objects $\varphi_s$ is \emph{strictly
  trivial} if $\varphi_s$ does not depend on $s$, i.e.,
  $\varphi_s=\varphi_{s_0}$ for some $s_0$. {It seems usual to say a
    bundle is trivial if it is trivializable.}
\end{itemize}

\section{Graph surgery}\label{s:graph-surgery}
We take an embedding $\Gamma\to {\mathrm{Int}\,X}$ of a labeled,
edge-oriented trivalent graph $\Gamma$. We put a Hopf link of the
spheres $S^{d-2}$ and $S^1$ at the middle of each edge, as in
Figure~\ref{fig:G-to-Y-link}.
\begin{figure}[!htbp]
\begin{center}
\includegraphics{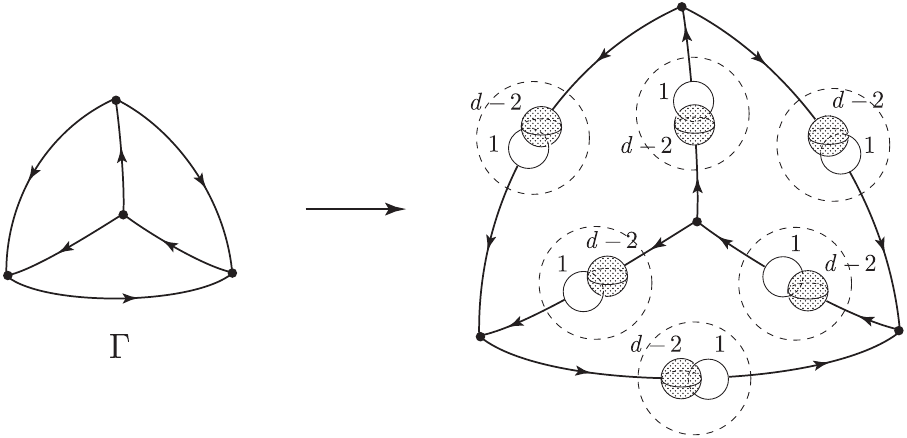}
\end{center}
\caption{Decomposition of embedded trivalent graph into Y-shaped
  pieces.}\label{fig:G-to-Y-link}
\end{figure}

\noindent
Then every vertex of $\Gamma$ gives a Y-shaped component
\emph{Y-graph} of Type I or and II, see Figure \ref{Y-graph-I-II}
below, i.e. an Y-graph is a vertex together with {framed} spheres
$S^{d-2}$ and $S^{1}$ {attached}.  We call the attached spheres {\it
  leaves} of a Y-graph.  This construction transforms the graph
$\Gamma$ into $2k$ components Y-graphs.
\begin{figure}[!htbp]
\begin{center}
\includegraphics[height=30mm]{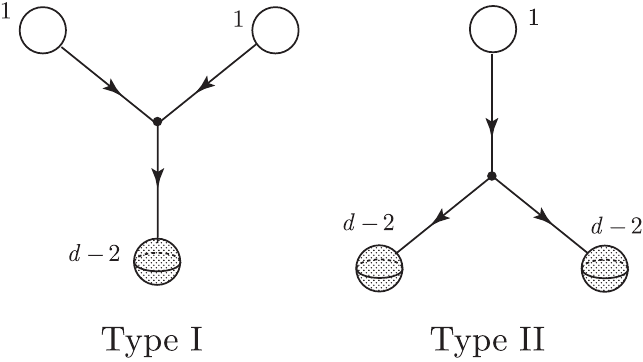}
\end{center}
\caption{Y-graphs of Type I and II}\label{Y-graph-I-II}
\end{figure}
We take small closed tubular neighborhoods of those Y-graphs, namely
the disjoint union of the $\ve$-tubular neighborhoods of the leaves
and the trivalent vertex (a point) connected by $\ve/2$-tubular
neighborhoods of the edges for some small $\ve$, and denote them by
$V^{(1)},V^{(2)},\ldots,V^{(2k)}$. They form a disjoint union of
handlebodies embedded in {$\mathrm{Int}\,X$}. A Type I Y-graph gives a
handlebody (of a Type I) which is diffeomorphic to the handlebody
obtained from {a} $d$-ball by attaching two 1-handles and one
$(d-2)$-handle in a standard way{, namely, along unknotted unlinked
  standard attaching spheres in the boundary of $D^d$}.  A Type II
Y-graph gives a handlebody (of a Type II) which is diffeomorphic to
the handlebody obtained from {a} $d$-ball by attaching one 1-handle
and two $(d-2)$-handles in a standard way.

Let $V=V^{(i)}$ be one of the Type I handlebodies and let
$\alpha_{\mathrm{I}}: S^0\to \Diff(\partial V)$, $S^0=\{-1,1\}$, be
the map defined by $\alpha_{\mathrm{I}}(-1)=\mathbbm{1}$, and by
setting $\alpha_{\mathrm{I}}(1)$ as the ``Borromean twist"
corresponding to the Borromean string link $D^{d-2}\cup D^{d-2}\cup
D^1\to D^d$. The detailed definition of $\alpha_{\mathrm{I}}$ can be
found in \cite[\S{4.5}]{Wa18b}.

Let $V=V^{(i)}$ be one of the Type II handlebodies and let
$\alpha_{\mathrm{II}}:S^{d-3}\to \Diff(\partial V)$ be the map defined
by comparing the {trivializations of} the family of complements
of an $S^{d-3}$-family of embeddings $D^{d-2}\cup D^1\cup
D^1\hookrightarrow D^d$ obtained by parametrizing the second component
in the Borromean string link $D^{d-2}\cup D^{d-2}\cup
D^1\hookrightarrow D^d$ with {that of} the trivial family {of
  $\partial V$}. The detailed definition of $\alpha_{\mathrm{II}}$ can
be found in \cite[\S{4.6}]{Wa18b}.

\begin{figure}[!htbp]
\setlength{\unitlength}{1mm}
  \begin{picture}(0,0)
    \put(-3,-5){{\small $\R^{d-2}$}}
    \put(28,-25){\textcolor{blue}{\small $D^{d-2}$}}
    \put(34,-48){{\small $\R^{1}$}}
    \put(34,-38){{\small $D^{1}$}}
 \put(-38,-41){{\small $\R^{1}$}}
    \put(-25,-14){\textcolor{red}{\small $D^{d-2}$}}
  \end{picture} 
\begin{center}
\includegraphics[height=60mm]{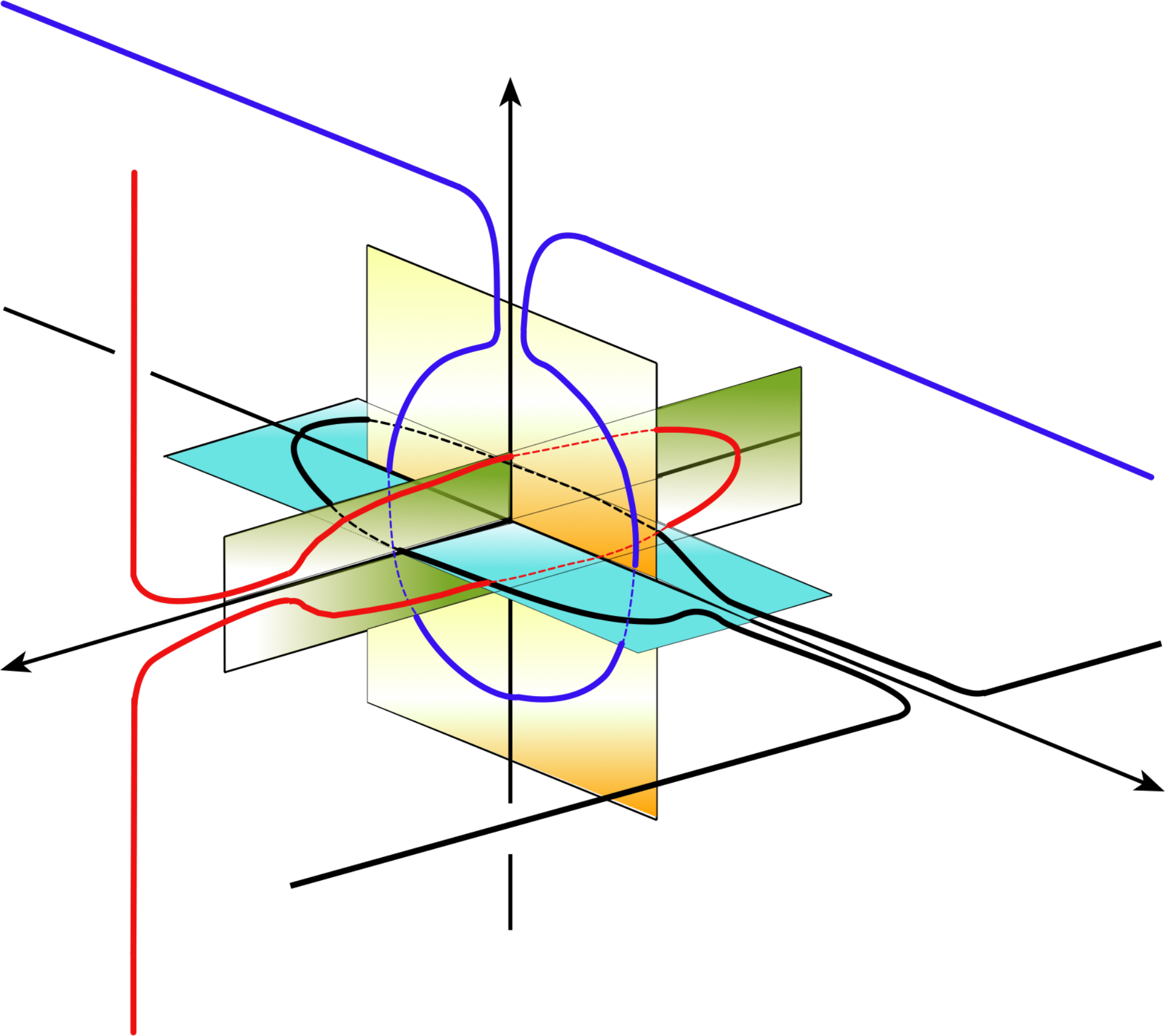}
\end{center}
  \caption{Borromean string link $D^{d-2}\cup D^{d-2}\cup
    D^1\hookrightarrow D^d$}\label{fig111a}
\end{figure}
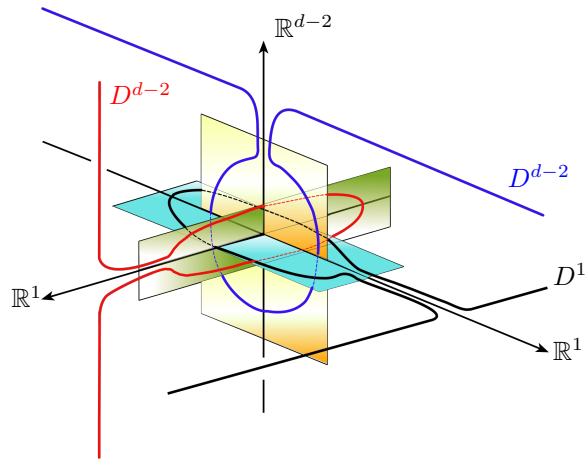
The Borromean string link has the following important property, which
will be used later.
\begin{Property}\label{propt:brunnian}
If one of the three components in the Borromean string link
$D^{d-2}\cup D^{d-2}\cup D^1\hookrightarrow D^d$ is deleted, then the
string link given by the remaining two components is isotopic relative
to the boundary to the standard inclusion of disks.
\end{Property}
For each $i$-th vertex of $\Gamma$ we let $K_i=S^0$ or $S^{d-3}$
depending on whether this vertex is of Type I or II.  Accordingly, let
$\alpha_i:K_i\to \Diff(\partial V^{(i)})$ be $\alpha_{\mathrm{I}}$ or
$\alpha_{\mathrm{II}}$. Let $B_\Gamma=K_1\times \cdots\times K_{2k}$.
By using the families of twists above, we define
\begin{equation*}
E^\Gamma=(B_\Gamma\times ({X}-\mathrm{Int}\,(V^{(1)}\cup\cdots\cup V^{(2k)})))\cup_\partial (B_\Gamma\times (V^{(1)}\cup\cdots\cup V^{(2k)})),
\end{equation*}
where the gluing map is given by
\begin{equation*}
\begin{split}
&\psi:B_\Gamma\times(\partial V^{(1)}\cup\cdots\cup\partial V^{(2k)})
  \to B_\Gamma\times(\partial V^{(1)}\cup\cdots\cup\partial V^{(2k)})\\
 &\psi(t_1,\ldots,t_{2k},x)=(t_1,\ldots,t_{2k},\alpha_i(t_i)(x))
  \quad \mbox{(for $x\in \partial V^{(i)}$)}.
\end{split}
\end{equation*}
\begin{Prop}[{\cite{Wa18b}}]\label{prop:v-framing}
Let $X$ be a $d$-dimensional compact manifold having a framing
$\tau_0$. The natural projection $\pi^\Gamma:E^\Gamma\to B_\Gamma$ is
{an $(X,\partial)$}-bundle, and it admits a vertical framing that is
compatible with the surgery and that agrees with $\tau_0$ near the
boundary, and it gives an element of
\begin{equation*}
  {\Omega_{(d-3)k}^{SO}(\widetilde{B\Diff}(X,\partial)),}
\end{equation*}
where $\widetilde{B\Diff}(X,\partial)$ is the classifying space for
framed $(X,\partial)$-bundles.
\end{Prop}
The following is a most precise statement
of our first main result, Theorem 1.3.
\begin{Thm}\label{thm:bordism-simplify}
Let $d\geq 4$ be an integer. Let $(m_1,m_2)=(1,d-2)$ if $d$ is even
and let $m_1=m_2=(d-1)/2$ if $d$ is odd.
\begin{enumerate}
\item The {$(X,\partial)$}-bundle $\pi^\Gamma:E^\Gamma\to B_\Gamma$
  for an embedding $\phi:\Gamma\to \mathrm{Int}\,{X}$ is related by
  {an $(X,\partial)$}-bundle bordism to {an $(X,\partial)$}-bundle
  $\varpi^\Gamma:\overline{E}^\Gamma\to S^{k(d-3)}$ obtained from the
  {product bundle $S^{k(d-3)}\times X \to S^{k(d-3)}$} by fiberwise
  {surgeries} along a $S^{k(d-3)}$-family of framed links
  $h_s:S^{m_1}\cup S^{m_2}\to \mathrm{Int}\,{X}$, $s\in S^{k(d-3)}$,
  that satisfies the following conditions:
\begin{enumerate}
\item $h_s$ is isotopic to the Hopf link for each $s$.
\item The restriction of $h_s$ to $S^{m_2}$ is a constant $S^{k(d-3)}$-family.
\item There is a small neighborhood $N$ of $\mathrm{Im}\,\phi$ such
  that the image of $h_s$ is included in $N$ for all $s\in
  S^{k(d-3)}$.
\end{enumerate}
\item There exists an $(X\times I,\partial_\sqcup
  (X\times I))$-bundle $\Pi^\Gamma: W^\Gamma_h\to S^{k(d-3)}$ such
  that
\begin{enumerate}
\item the fiberwise restriction of $\Pi^\Gamma$ to {$X\times \{1\}$}
  is $\varpi^\Gamma$,
\item the manifold $W^\Gamma_h$ is obtained by attaching
  $S^{k(d-3)}$-families of $(d+1)$-dimensional $m_1$- and
  $(m_1+1)$-handles to the product {$(X\times I)$}-bundle
  {$S^{k(d-3)}\times (X\times I)\to S^{k(d-3)}$} at
  {$S^{k(d-3)}\times (X\times\{1\})$}.
\end{enumerate}
\end{enumerate}
\end{Thm}
For concreteness, we prove
  Theorem~\ref{thm:bordism-simplify} in the case of even $d$ in
  Corollary~\ref{cor:bordism} and
  Proposition~\ref{prop:pseudoisotopy}. The case of odd $d$ is
  completely analogous, with the following replacements:
\begin{itemize}
\item Y-surgery is given by an $S^{m-1}$-family of embeddings $D^m\cup
  D^m\cup D^m\to D^{2m+1}$ obtained by parametrizing a Borromean
  string link $ D^{2m-1}\cup D^m\cup D^m\to D^{2m+1}.$ We take this at
  each trivalent vertex, so $B_\Gamma=S^{m-1}\times
  S^{m-1}\times\cdots\times S^{m-1}$ ($2k$ factors).
\item In the proof of an analogue of
  Proposition~\ref{prop:pseudoisotopy}, we replace a trivial family of
  $(m+1)$-handles with that of $m$-handles.
\end{itemize}

\section{Alternative definition of Y-surgery by
  framed links}\label{s:Ysurg-link}
\subsection{Framed link for Type I surgery}
Let $d\geq 4$. Let $K_1,K_2,K_3$ be the unknotted spheres {in
  $\mathrm{Int}\,X$} that are parallel to the cores of the handles of
Type I handlebody $V$ of indices 1,1,$d-2$, respectively. Let $c_i$ be
a small unknotted sphere in $\mathrm{Int}\,V$ that links with $K_i$
with the linking number 1. Let $L_1'\cup L_2'\cup L_3'$ be a Borromean
rings of dimensions $d-2,d-2,1$ embedded in a small ball in
$\mathrm{Int}\,V$ that is disjoint from
\begin{equation*}
K_1\cup K_2\cup K_3\cup
c_1\cup c_2\cup c_3.
\end{equation*}
For each $i=1,2,3$, let $L_i$ be a knotted sphere in $\mathrm{Int}\,V$
obtained by connect summing $c_i$ and $L_i'$ along an embedded arc
that is disjoint from the cocores of the 1-handles and from other
components, so that $L_i$'s are mutually disjoint. Then $K_i\cup L_i$
is a Hopf link in $d$-dimension. If $K_i$ is null in $X$ for
$i=1,2,3$, namely, $K_i$ bounds an embedded disk in $X$, then each
component of the six component link $\bigcup_{i=1}^3 (K_i\cup L_i)$ is
an unknot in $X$, and we may consider it as a framed link by canonical
framings {induced from the standard sphere by the isotopies along the
  spanning disks} (Figure~\ref{fig:theta-surgery}, $V^{(1)}$). The
following is a framed link definition of Type I surgery.
\begin{Def}[Y-surgery of Type I]\label{def:type-I}
We define the Type I surgery on $V$ to be the surgery along the six
component framed link $\bigcup_{i=1}^3 (K_i\cup L_i)$ in $V$.
\end{Def}
We will see in section \ref{ss:Y-type-I} (Remark~\ref{rem:type-I})
that this definition is equivalent to that we have given in section
\ref{s:graph-surgery}. For $d=3$, this equivalence was proven by
S.~Matveev in \cite{Mat}. Our proof is an analogue of \cite[\S{2}]{Ha}
or \cite[Lemma~2.1]{GGP}.

\subsection{Family of framed links for Type II surgery}
Similarly, let $K_1,K_2,K_3$ be the unknotted spheres {in
  $\mathrm{Int}\,X$} that are parallel to the cores of the handles of
Type II handlebody $V$ of indices 1,$d-2$,$d-2$, respectively. Let
$c_i$ be a small {framed} unknotted sphere in $\mathrm{Int}\,V$ that
links with $K_i$ with the linking number 1. Let $L_{1,s}'\cup
L_{2,s}'\cup L_{3,s}'$ ($s\in S^{d-3}$) be a $(d-3)$-parameter family
of three component {framed} links of dimensions $d-2,1,1$ with only
{(isotopically)} unknotted components embedded in a small ball in
$\mathrm{Int}\,V$ disjoint from $K_1\cup K_2\cup K_3\cup c_1\cup
c_2\cup c_3$ such that $L_{1,s}', L_{3,s}'$ are unknotted components
in $V$ that do not depend on $s$, and the union of the locus of
$L_{2,s}'$ and $L_{1,s}'\cup L_{3,s}'$ forms a closure of the
Borromean string link of dimensions $d-2,d-2,1$.

For each $i=1,2,3$, let $L_{i,s}$ be a knotted sphere in
$\mathrm{Int}\,V$ obtained by connect summing $c_i$ and $L_{i,s}'$
along an embedded arc that is disjoint from the cocores of the
1-handles and from other components, so that $L_{i,s}$'s are mutually
disjoint. Then $K_i\cup L_{i,s}$ is a Hopf link in $d$-dimension. If
$K_i$ is null in $X$ for $i=1,2,3$, then each component of the six
component link $\bigcup_{i=1}^3 (K_i\cup L_{i,s})$ is fiberwise
isotopic to a constant family of an unknot in $X$, and we may consider
it as a family of framed links by canonical framings
(Figure~\ref{fig:theta-surgery}, $V^{(2)}$). The following is a framed
link definition of Type II surgery.
\begin{figure}[!htbp]
\begin{center}
\includegraphics{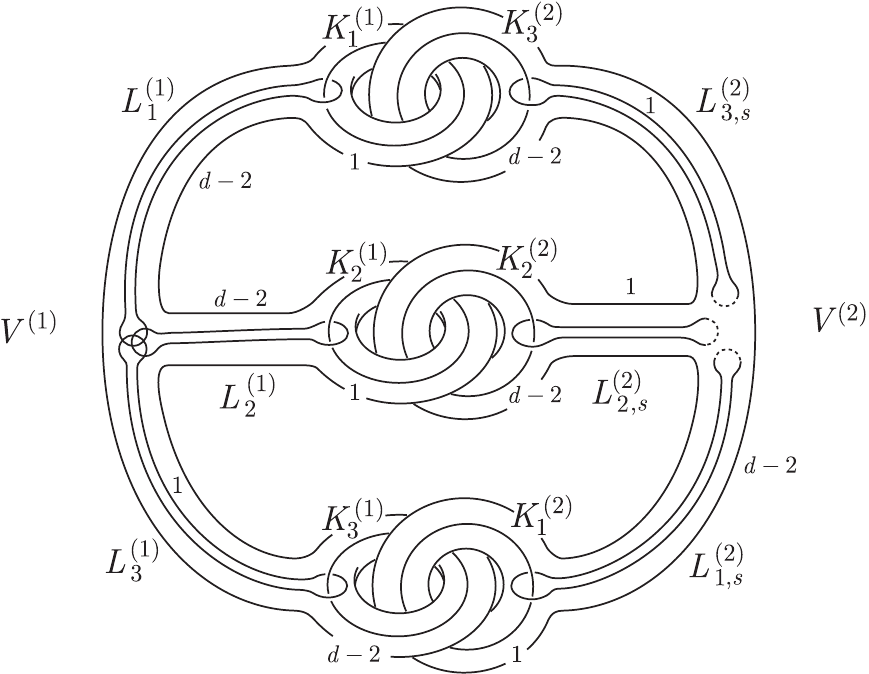}
\end{center}
\caption{Framed link for $\Theta$-graph surgery}\label{fig:theta-surgery}
\end{figure}
\begin{Def}[Y-surgery of Type II]\label{def:type-II}
We define the Type II surgery on $V$ to be the $S^{d-3}$-family of
surgeries along the family of the six component framed link
$\bigcup_{i=1}^3 (K_i\cup L_{i,s})$, $s\in S^{d-3}$, in $V$, which
produces a $(V,\partial)$-bundle over $S^{d-3}$.
\end{Def}
We will see in section \ref{ss:Y-type-II}
(Remark~\ref{rem:type-II}) that this definition is equivalent to that
we have given in section \ref{s:graph-surgery}.

\subsection{Hopf link surgery for links}
We would like to describe the effect of a Y-surgery of Type I or II
when a link in the complement of the Y-graph is present. Since a
Y-surgery consists of surgeries of three Hopf links, we shall first
consider the effect of a single Hopf link surgery.

\subsubsection{Surgery of $X$ on a framed link $L$}\label{ss:link-surgery}
We shall recall the definition of surgery on a framed link $L$ in
$\mathrm{Int}\,X$. Let $W=X\times I$ and let $W^L$ be a
$(d+1)$-dimensional cobordism obtained from $W$ by attaching disjoint
handles along $L\times\{1\}$ in $X\times\{1\}$. In more detail, for
each framed embedding $\ell:S^i\hookrightarrow L\times\{1\}$, we
attach $(d+1)$-dimensional $(i+1)$-handle along a small tubular
neighborhood of $\ell$. The handle attachments can be done disjointly
and simultaneously, and gives a $(d+1)$-dimensional cobordism $W^L$
between $X\times\{0\}$ and some $d$-manifold $X^L$. We say that $W^L$
is obtained from $X\times I$ by surgery along $L$, or by attaching
handles along $L$.  Let $\partial_\sqcup W^L=X\times\{0\}\cup \partial
X\times I$, $\partial_-W^L=X\times\{0\}$, and $\partial_+W^L=\partial
W^L\setminus\mathrm{Int}\,\partial_\sqcup W^L=X^L$.
\begin{figure}[!htbp]
\begin{center}
\includegraphics[height=30mm]{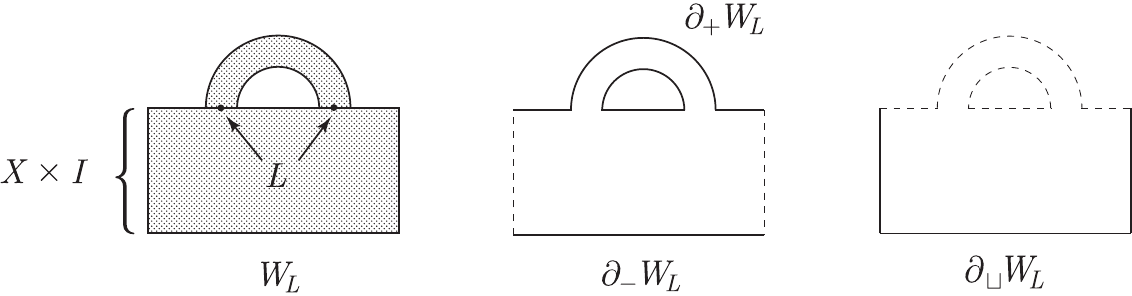}
\end{center}
\caption{Manifold $W^L$}\label{fig:handles}
\end{figure}
Here, $W$ may be more general cobordism. Namely, let $W$ be a
  relative cobordism between $\partial_-W=X\times\{0\}$ and some
  another manifold $\partial_+ W$ such that
\begin{equation}\label{eq:d-cobordism}
  \partial W=\partial_+W\cup_{\partial
  X\times\{1\}} (\partial X\times I)\cup_{\partial X\times\{0\}}
  \partial_-W.
\end{equation}
    Let $\partial_\sqcup W=(\partial X\times I)\cup_{\partial
      X\times\{0\}} \partial_-W$. For a framed link $L$ in
    $\partial_+W$, we define the manifold $W^L$ as a
    $(d+1)$-dimensional cobordism obtained from $W$ by attaching
    disjoint handles along $L$ in $\partial_+W$. A {\it
      $(W,\partial_\sqcup W)$-bundle} is defined as a $W$-bundle with
    structure group $\Diff(W,\partial_\sqcup W)$. Concordance between
    two $(W,\partial_\sqcup W)$-bundles can be defined similarly as for
    $W=X\times I$.

\subsubsection{Concordance of a cobordism}
When a link $c$ in $X$ is present, surgery on a framed link $L$ in
$X\setminus c$ changes the pair $(X,c)$. It may happen that surgeries
for two choices $(c,L)$ and $(c',L')$ of the links in $X$ should be
considered equivalent. Here, we consider the notion of concordance
between two such data, defined as follows.
\begin{Def}\label{def:concordance1}
Let $W$ be a relative cobordisms between
  $\partial_-W=X\times\{0\}$ and $\partial_+ W$ satisfying
  (\ref{eq:d-cobordism}) above, and let $W'$ be another such relative
  cobordism such that $W$ and $W'$ are concordant as
  $(W,\partial_\sqcup W)$-bundles over points. For framed links $L$
and $L'$ in $\partial_+W$ and $\partial_+W'$, respectively, we say
that the pairs $(W,L)$ and $(W',L')$ are {\it concordant} if $L\sqcup
L'$ is the restriction of a trivialized (fiberwise normally) framed
subbundle $\widetilde{L}$ of the top boundary of a concordance between
the cobordisms $W$ and $W'$.
\end{Def}
\begin{figure}[!htbp]
\begin{equation*}
\includegraphics[height=40mm]{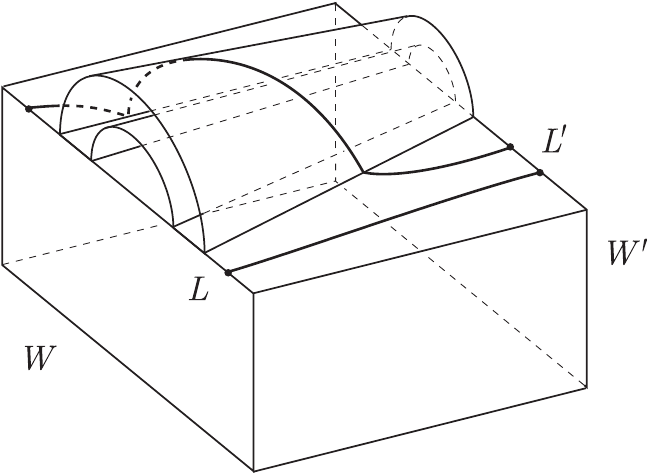}
\end{equation*}
\caption{A concordance between $(W,L)$ and
  $(W',L')$}
\end{figure}
\begin{Rem}
The definition of concordance for manifold pair is not as
usual. Usually, the projection $\mathrm{proj}\circ
q|_{\widetilde{L}}:\widetilde{L}\to I$ for the concordance
$\widetilde{L}$ may not be level-preserving, whereas we assume so.  It
is evident from definition that if $(W,L)$ and $(W',L')$ are
concordant, the cobordisms $W^L$ and $W'^{L'}$ are concordant.
\end{Rem}
\subsubsection{Hopf link surgery for links}
Suppose a $d$-manifold $X$ is equipped with some embedded objects
inside, such as links or Y-links. By a {\it small Hopf link} in $X$,
we mean a Hopf link in a $d$-ball $b$ in $X$ with sufficiently small
radius so that $b$ is disjoint from the given embedded objects in $X$.

Let $K,L$ be the components of a Hopf link in $\mathrm{Int}\,X$ of
dimensions $1,d-2$ with standard framing and with spanning disks
$d_1,d_2$ in $\mathrm{Int}\,X$, respectively.  Let $c_1,c_2$ be
{framed} spheres of dimensions $d-2,1$, respectively, in
$\mathrm{Int}\,X$ such that $d_1$ (resp. $d_2$) intersects $c_1$
(resp. $c_2$) transversally by one point and does not intersect other
component in $c_1\cup c_2$ nor $K\cup L$ (See
Figure~\ref{fig:Hopf-K-L}, left).  Let $N_{d_1\cup d_2}$ be a small
closed neighborhood of $d_1\cup d_2$.  Let $c_1'\cup c_2'$ be a framed
link in $\mathrm{Int}\,X$ obtained from $c_1\cup c_2$ by
component-wise connect-summing a small Hopf link in $N_{d_1\cup d_2}$.
Let $K'\cup L'$ be another {framed} Hopf link in $N_{d_1\cup d_2}$ that
is small and disjoint from $c_1'\cup c_2'$. (See
Figure~\ref{fig:Hopf-K-L}, right.) {The following lemma is an analogue
  of \cite[Proposition~2.2]{Ha}.}
\begin{Lem}[Hopf link surgery]\label{lem:Hopf-surgery}
Let $W=X\times [a,b]$, where we identify $X\times\{b\}$ with
  $X$. Then the pairs $(W^{K\cup L},c_1\cup c_2)$ and $(W^{K'\cup
    L'},c_1'\cup c_2')$ are concordant. Moreover, we may assume that
  the concordance is strictly trivial on
  $(X\setminus\mathrm{Int}\,N_{d_1\cup d_2})\times[a,b]$.
\end{Lem}
\begin{figure}[!htbp]
\includegraphics[height=25mm]{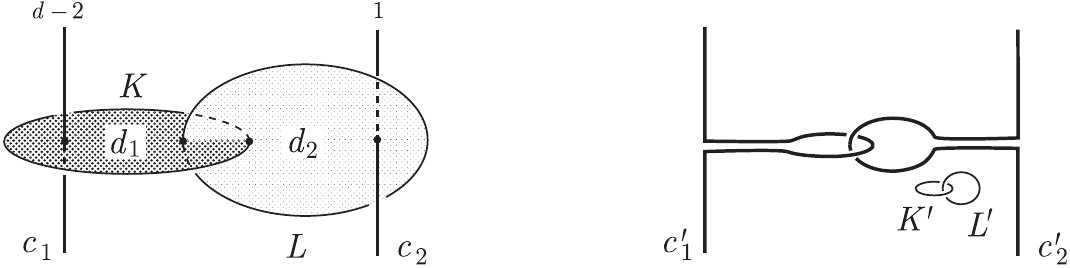}
\caption{Surgery along Hopf link $K\cup L$}\label{fig:Hopf-K-L}
\end{figure}
This Lemma can also be applied to a general cobordism $W$ with
$\partial_+W=X$ by considering a collar neighborhood of $\partial_+W$
as $X\times[a,b]$.
\begin{proof}
Let $L_1=L$ and $L_2=K$. {Before} going to the proof, we define
band-sums $c_i\#L_i$. We choose an embedded path $\gamma_i$ in $d_i$
that goes from $d_i\cap c_i$ to $d_i\cap L_i$. Then we may connect-sum
$c_i$ with $L_i$ along $\gamma_i$ so that the result is disjoint from
$d_i$. More precisely, the restriction of the normal bundle of $d_i$
on $\gamma_i$ is an $\R^{\dim{c_i}}$-bundle. Thus $\gamma_i$ can be
thickened to a $\dim{c_i}$-disk bundle in $N_{d_1\cup d_2}$ that is
perpendicular to $d_i$ and its restriction on the endpoints are
$\dim{c_i}$-disks in $c_i$ and $L_i$. The disk bundle is a
$(\dim{c_i}+1)$-dimensional 1-handle attached to $c_i\cup L_i$ along
which surgery can be performed. This surgery produces the connected
sum $c_i\# L_i$ along $\gamma_i$ whose result is disjoint from
$d_i$. See Figure~\ref{fig:slide-path}, left.

Now {let us return to the framed link $K\cup L\cup c_1\cup c_2$}.  We
perform surgeries on the link $K\cup L=L_2\cup L_1$, then the
component $c_i$ can be slid over $\gamma_i$ and the
$(\dim{L_i}+1)$-handle attached to $L_i$. The result of the handle
slide is $c_i\#\overline{L}_i$ defined as in the previous paragraph,
where $\overline{L}_i$ is a parallel copy of $L_i$ obtained from $L_i$
by slightly pushing off by one direction of the framing on $L_i$. We
denote by $c_i''$ the resulting {framed} sphere
$c_i\#\overline{L}_i$. We assume that $c_i''\setminus c_i$ is included
in $N_{d_1\cup d_2}$ and agrees with $c_i$ outside $N_{d_1\cup
  d_2}$. The link $c_1''\cup c_2''$ is obtained from $c_1\cup c_2$ by
component-wise connect-summing Hopf links in $N_{d_1\cup d_2}$, which
is realized by handle slides. Thus
\begin{equation*}
(W^{K\cup L},c_1\cup c_2) \ \ \ \mbox{and} \ \ \
(W^{K\cup L},c_1''\cup c_2'')
\end{equation*}
are concordant.

We need to show that $(c_1''\cup c_2'')\cup (K\cup L)$ is isotopic in
$N_{d_1\cup d_2}$ to $(c_1'\cup c_2')\cup (K'\cup L')$. Since $c_1''$
is disjoint from $d_1$, the component $K$ can be shrinked along $d_1$
to a small sphere $K'$ in a small $d$-disk $b$ around the point
$d_1\cap L$ without intersecting $c_1''\cup c_2''$ as in
Figure~\ref{fig:slide-path}. Similarly, since $c_2''$ is disjoint from
$d_2$, the component $L$ can be shrinked along $d_2$ to a small sphere
$L'$ in $b$, without intersecting $c_1''\cup c_2''$, so that $K'\cup
L'$ is a small Hopf link in $b$.

Then similar isotopy can be performed for $c_1''\cup c_2''$ in
$N_{d_1\cup d_2}$ so that the part parallel to $L_1$ and $L_2$ is
shrinked to a small Hopf link with bands. We may assume that this
isotopy is disjoint from $b$. The result of the deformation is
$(c_1'\cup c_2')\cup (K'\cup L')$. Thus, the deformations performed so
far give a desired concordance.
\begin{figure}[!htbp]
\begin{equation*}
\includegraphics[height=25mm]{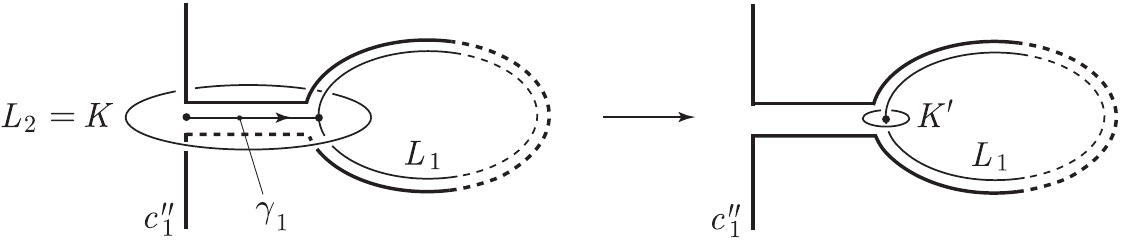}
\end{equation*}
\caption{Sliding $c_1$ along $\gamma_1$, and then isotoping $K$ to $K'$. }\label{fig:slide-path}
\end{figure}
\end{proof}

\begin{Rem}\label{rem:family-hopf}
The framed Hopf link $K\cup L$ may be replaced by some {``smooth
  family''} of {framed} Hopf links. More precisely, let $K_s\cup L_s$
be a smooth family of framed Hopf links parametrized over a compact
connected manifold $B$ with a base point $s_0$, such that
\begin{itemize}
\item[(a)] $L_s=L$, and hence $L_s$ bounds $d_2$.
\item[(b)] $d_2$ intersects 1-dimensional arc in $c_2$
  transversally by one point.
\item[(c)] $K_s$ bounds a smooth family of disks $d_{1,s}$ in
  $\mathrm{Int}\,X$ such that for each $s$, $L_s$ intersects $d_{1,s}$
  transversally by one point, and $K_s$ intersects $d_2$ transversally
  by one point.
\item[(d)] $d_{1,s}$ and $d_{1,s_0}$ agree on a neighborhood of
  the arc $d_{1,s_0}\cap d_2$.
\end{itemize}
Then surgery on the family $K_s\cup L_s$ gives a family of cobordisms
that is concordant (in the sense of
Definition~\ref{def:concordance-bundle}) to the strictly trivial family of
cobordisms with a nontrivial family of spheres $c_{2,s}'$ on the top,
which is obtained from $c_2$ by connected-summing with parallel copies
of $K_s$.
\begin{figure}[!htbp]
\begin{equation*}
\includegraphics[height=28mm]{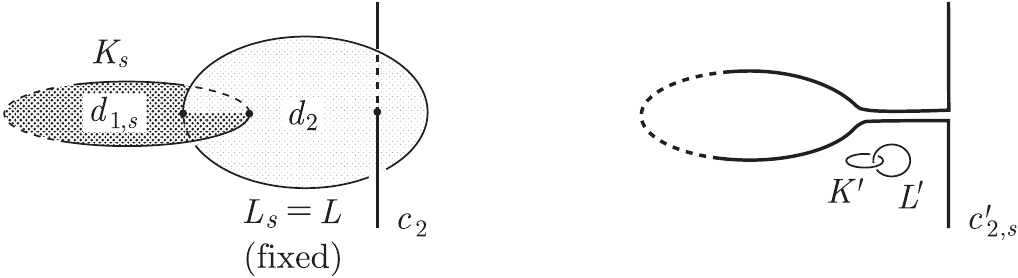}
\end{equation*}
\caption{Surgery along a family of Hopf links $K_s\cup L$.}\label{fig:hopf-Ks-L}
\end{figure}

For example, if $B=S^{d-3}$, then the family $K_s$ may be chosen so
that the associated map $B\times D^2\to B\times \mathrm{Int}\,X$ for
the spanning disks $d_{1,s}$ intersects $B\times c_1$ transversally by
one point in $B\times \mathrm{Int}\,X$. Such a family of framed Hopf
links surgery will play an important role in the framed link description
of the Type II surgery in Lemma~\ref{lem:Y-surgery-link-II}.
\end{Rem}

\subsection{Type I Y-surgery for links}\label{ss:Y-type-I}
We say that a leaf $\ell$ of a Y-graph $T$ is {\it simple} relative to
a submanifold $c$ in $\mathrm{Int}\,X$ with $\dim{\ell}+\dim{c}=d-1$,
if the following conditions are satisfied.
\begin{enumerate}
\item The leaf $\ell$ bounds a disk $m$ in $\mathrm{Int}\,X$. 
\item The disk $m$ intersects $c$ transversally by one point.
\end{enumerate}
See Figure~\ref{fig:simple-Y-graph} (a). We say that a Y-graph $T$
with leaves $\ell_1,\ell_2,\ell_3$ is {\it simple} relative to a three
component link $c_1\cup c_2\cup c_3$ in $\mathrm{Int}\,X$ with
$\dim{\ell}_i+\dim{c_i}=d-1$, if the following conditions are
satisfied.
\begin{enumerate}
\item The leaves $\ell_1,\ell_2,\ell_3$ bound disjoint disks
  $m_1,m_2,m_3$ in $\mathrm{Int}\,X$, respectively.
\item For each $i$, the disk $m_i$ intersects $c_i$ transversally by
  one point and does not intersect other components in $c_1\cup
  c_2\cup c_3$.
\end{enumerate}
See Figure~\ref{fig:simple-Y-graph} (b). In this case, we take a small
closed neighborhood of $T\cup m_1\cup m_2\cup m_3$ that is a $d$-disk
and denote it by $N(T)$.
\begin{figure}[!htbp]
\[ \includegraphics[height=33mm]{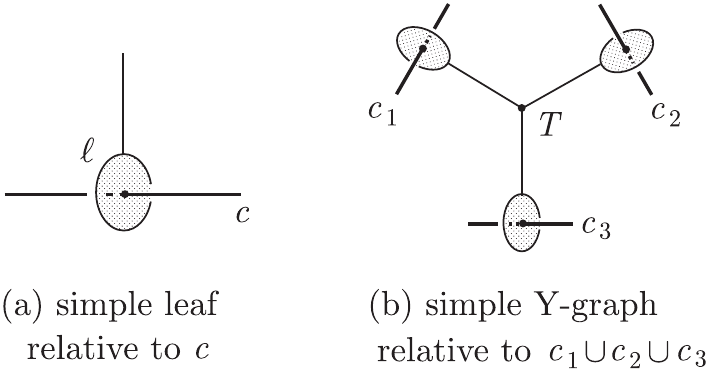} \]
\caption{}\label{fig:simple-Y-graph}
\end{figure}

Let $L_T$ be the framed link associated to $T$, as in
Definition~\ref{def:type-I}. We define $W^T$ as $W^{L_T}$ in the sense
of section \ref{ss:link-surgery}.
\begin{Lem}[Type I surgery]\label{lem:Y-surgery-link-I}
Let $W=X\times[a,b]$, where we identify $X\times\{b\}$ with $X$.
Suppose that the leaves of a Y-graph $T$ of Type I in
$\mathrm{Int}\,X$ of dimensions $1,1,d-2$ are linked to {framed}
submanifolds $c_1,c_2,c_3$ of dimensions $d-2,d-2,1$, respectively,
and that $T$ is simple relative to $c_1\cup c_2\cup c_3$.  Let
$c_1'\cup c_2'\cup c_3'$ be a framed link that is obtained from
$c_1\cup c_2\cup c_3$ by component-wise connect-summing Borromean
rings in $N(T)$. 
\begin{enumerate}
\item There are three disjoint small Hopf links $h_1,h_2,h_3$ in
  $N(T)\setminus(c_1'\cup c_2'\cup c_3')$ and a concordance between
  the pairs $(W^T, c_1\cup c_2\cup c_3)$ and $(W^{h_1\cup h_2\cup
    h_3},c_1'\cup c_2'\cup c_3')$ that is strictly trivial on
  $(X\setminus\mathrm{Int}\,N(T))\times [a,b]$.
\item Moreover, if we consider up to isotopy, we may assume that two
  of the components of the link $c_1'\cup c_2'\cup c_3'$ agree as subsets of
  $\mathrm{Int}\,X$ with those of $c_1\cup c_2\cup c_3$.
\end{enumerate}
\end{Lem}
\begin{figure}[!htbp]
\[ \includegraphics[height=35mm]{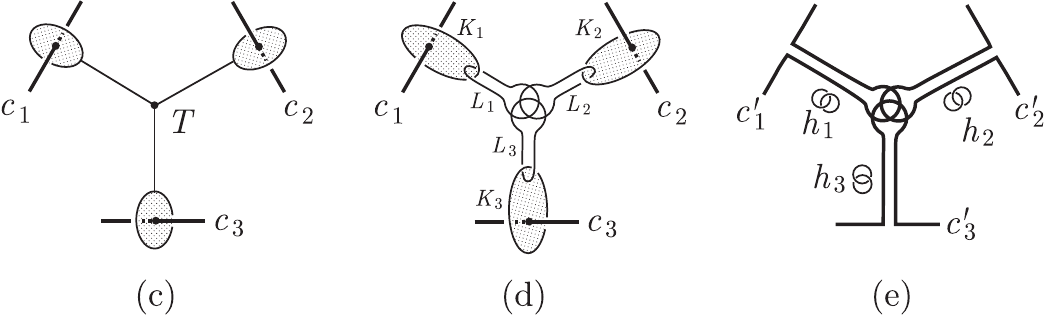} \]
\caption{}\label{fig:Y-graph-for-link}
\end{figure}
\begin{Rem}\label{rem:type-I}
Lemma~\ref{lem:Y-surgery-link-I} shows that the two definitions of
Type I surgeries: ``the complement of thickened string link'' given in
section \ref{s:graph-surgery}, and ``framed link surgery'' in
Definition~\ref{def:type-I} are equivalent. Namely, let $L$ be the six
component framed link of Definition~\ref{def:type-I} in $V$. Then the
latter definition is given by surgery along $L$ in $V$. According to
Lemma~\ref{lem:Y-surgery-link-I} and if we consider modulo small Hopf
links, this surgery replaces $V$ with another one that is obtained by
taking the complements of the Borromean string link. The relative
diffeomorphism type of the resulting manifold is determined uniquely
and agrees with the former definition of Type I surgery.
\end{Rem}
\begin{proof}[Proof of Lemma~\ref{lem:Y-surgery-link-I}]
Lemma~\ref{lem:Y-surgery-link-I} is obtained by iterated applications
of the concordance deformations of
Lemma~\ref{lem:Hopf-surgery}. Namely, by Definition~\ref{def:type-I},
the surgery on $T$ is given by the surgery on a six component framed
link $ \bigcup_{i=1}^3(K_i\cup L_i).  $ Since $T$ is simple relative
to $c_1\cup c_2\cup c_3$, the component $K_i$ bounds a disk $d_i$ in
$N(T)$. After relabeling if necessary, we may assume that for each
$i$, the intersections $d_i\cap c_i$ and $d_i\cap L_i$ are both one
point and orthogonal, and $d_i$ does not {have other intersections
  with other link components}. (See Figure~\ref{fig:Y-graph-for-link}
(d).) We choose an embedded path $\gamma_i$ in $d_i$ that goes from
$d_i\cap c_i$ to $d_i\cap L_i$. Then we may {define the band sum
  $c_i\# L_i$} along $\gamma_i$ so that the result is disjoint from
$d_i$, as in the proof of Lemma~\ref{lem:Hopf-surgery}.

After performing surgeries on the framed link $\bigcup_{i=1}^3(K_i\cup
L_i)$, the component $c_i$ can be slid over $\gamma_i$ and then over
the $(\dim{L_i}+1)$-handle attached to $L_i$. The result of the handle
slide is $c_i\#\overline{L}_i$, where $\overline{L}_i$ is a parallel
copy of $L_i$ obtained from $L_i$ by slightly pushing off by one
direction of the framing on $L_i$. We define $c_i'$ as the resulting
{framed} sphere $c_i\#\overline{L}_i$. We assume that $\overline{L}_i$
is included in $N(T)$ and that $c_i'$ agrees with $c_i$ outside
$N(T)$. Now a {framed} link $c_1'\cup c_2'\cup c_3'$ has been obtained
from $c_1\cup c_2\cup c_3$ by sliding components over the handles
attached to $\bigcup_{i=1}^3(K_i\cup L_i)$, and also can be obtained
by component-wise connect-summing Borromean rings in $N(T)$. (See
Figure~\ref{fig:Y-graph-for-link} (e).)

We need to show that the Hopf links $K_i\cup L_i$ can be deformed into
a small Hopf link $h_i$. Since $c_1'$ is disjoint from $d_1$, the
component $K_1$ can be shrinked along $d_1$ to a small sphere $K_1'$
in a small $d$-disk around the point $d_1\cap L_1$, without
intersecting $c_1'\cup c_2'\cup c_3'$ during the shrinking
isotopy. Then by sliding other components $K_j\cup L_j$ over $K_1'$
for $j\neq 1$, the component $L_1$ can be made unlinked from $K_j\cup
L_j$. This slide does not change the isotopy type of
\begin{equation*}
(K_2\cup
L_2)\cup (K_3\cup L_3)\cup c_1'\cup c_2'\cup c_3'
\end{equation*}
  in $N(T)$, though does change that in $N(T)\setminus(K_1'\cup
L_1)$. Now the Hopf link $K_1'\cup L_1$ can be shrinked into a small
Hopf link $h_1$ without affecting other components. After that,
similar slidings can be performed for the Hopf links $K_2\cup L_2$ and
$K_3\cup L_3$ so that they can be separated and shrinked into disjoint
small Hopf links $h_2,h_3$, respectively.  Thus the deformations
performed so far consist of isotopy and slides over handles, which
give a desired concordance as in (1). The condition (2) follows from
Property~\ref{propt:brunnian}.
\end{proof}

\subsection{Type II Y-surgery for links}\label{ss:Y-type-II}
We shall give an analogue of Lemma~\ref{lem:Y-surgery-link-I} for Type
II Y-surgery, which is Lemma~\ref{lem:Y-surgery-link-II}.

\begin{Def}\label{def:concordance-bundle}
Let $W$ be as in Definition~\ref{def:concordance1} and let
$\pi_i:E_i\to B$ ($i=0,1$) be $(W,\partial_\sqcup W)$-bundles.  Let
$\widetilde{L}_0$ and $\widetilde{L}_1$ be fiberwise framed
trivialized subbundles of the top boundaries of $E_0$ and $E_1$,
respectively. We say that the pairs $(\pi_0,\widetilde{L}_0)$ and
$(\pi_1,\widetilde{L}_1)$ are {\it concordant} if
$\widetilde{L}_0\sqcup \widetilde{L}_1$ is the restriction of a
trivialized (fiberwise normally) framed subbundle $\widetilde{L}_{01}$
of the top boundary of a concordance between $\pi_0$ and
$\pi_1$. Let $\pi_0^{\widetilde{L}_0}:E_0^{\widetilde{L}_0}\to B$
  be the $(W^{L_0},\partial_\sqcup W^{L_0})$-bundle obtained by attaching a
  trivialized $B$-family of handles along $\widetilde{L}_0$, where
  $L_0$ is a fiber of $\widetilde{L}_0$.
\end{Def}
\begin{Lem}[Type II surgery]\label{lem:Y-surgery-link-II}
Let $W=X\times[a,b]$, where we identify $X\times\{b\}$ with $X$, and
let $\pi_0:S^{d-3}\times W \to S^{d-3}$ be the product $W$-bundle.
Suppose that the leaves of a Y-graph $T$ of Type II in
$\mathrm{Int}\,X$ of dimensions $1,d-2,d-2$ are linked to {framed}
submanifolds $c_1,c_2,c_3$ of dimensions $d-2,1,1$, respectively, and
that $T$ is simple relative to $c_1\cup c_2\cup c_3$.  Let $c_1'\cup
c_2'\cup c_3'$ be a framed trivialized subbundle of $S^{d-3}\times
\mathrm{Int}\,X\to S^{d-3}$ that is obtained from
$S^{d-3}\times(c_1\cup c_2\cup c_3)\hookrightarrow S^{d-3}\times
\mathrm{Int}\,X$ by fiberwise component-wise connect-summing
$S^{d-3}$-family of framed links $S^{d-2}\cup S^1\cup S^1\to N(T)$
that defines the Type II surgery.
\begin{enumerate}
\item There is a concordance 
  between the pairs of $(W,\partial_\sqcup W)$-bundles over $S^{d-3}$
\begin{equation*}
  (\pi_0^T,S^{d-3}\times(c_1\cup c_2\cup c_3)) \ \ \mbox{and}
  \ \ (\pi_0^{h_1\cup h_2\cup h_3},c_1'\cup c_2'\cup c_3')
\end{equation*}
that is strictly trivial on $(X-\mathrm{Int}\,N(T))\times [a,b]$.
\item We may assume that two of the components of $c_1'\cup c_2'\cup
  c_3'$ agree with those of the inclusion 
  $S^{d-3}\times (c_1\cup c_2\cup c_3)\hookrightarrow
  S^{d-3}\times \mathrm{Int}\,X$.
\end{enumerate}
\end{Lem}
\begin{Rem}\label{rem:type-II}
Lemma~\ref{lem:Y-surgery-link-II} shows that the two definitions of
Type II surgeries given in section \ref{s:graph-surgery} and
Definition~\ref{def:type-II} are equivalent, as in
Remark~\ref{rem:type-I}.
\end{Rem}
\begin{proof}[Proof of Lemma~\ref{lem:Y-surgery-link-II}]
Proof is analogous to that of Lemma~\ref{lem:Y-surgery-link-I} and can
be done by iterated applications of the concordance deformations of
Remark~\ref{rem:family-hopf}. We only need to replace $L_i$ {in
  Lemma~\ref{lem:Y-surgery-link-I}} with a family of links $L_{i,s}$
in $N(T)$, $s\in S^{d-3}$. By Definition~\ref{def:type-II}, the
surgery on $T$ is given by the surgery on a six component link
\begin{equation*}
\begin{array}{c}
  \bigcup_{i=1}^3(K_i\cup L_{i,s})
\end{array}
\end{equation*}
  in each fiber over $s\in S^{d-3}$. We assume that for all $s$,
  $L_{i,s}$ agrees with $L_{i,s_0}$ near the base point of
  $L_{i,s_0}$.  Then $K_i\cup L_{i,s}$ satisfies the conditions
  (a)--(d) of Remark~\ref{rem:family-hopf}. Then $c_1'\cup c_2'\cup
  c_3'$ is defined by fiberwise component-wise connected-summing the
  family $L_{1,s}\cup L_{2,s}\cup L_{3,s}$ of framed links to
  $S^{d-3}\times (c_1\cup c_2\cup c_3)$. The proofs of (1) and
  (2) are parallel to those for Lemma~\ref{lem:Y-surgery-link-I}.
\end{proof}
The following lemma is an analogue of Habiro's move 10
(\cite[Proposition~2.7]{Ha}).
\begin{Lem}[Y-graph with Null-leaf]\label{lem:null-leaf-I}
Let $W=X\times[a,b]$, where we identify $X\times\{b\}$ with $X$.
Suppose that the leaves $\ell_1,\ell_2,\ell_3$ of a Y-graph $T$ of
Type I or II in $\mathrm{Int}\,X$ bound disjoint disks $m_1,m_2,m_3$ in
$\mathrm{Int}\,X$, respectively. Suppose there are disjoint
submanifolds $c_1,c_2$ in $\mathrm{Int}\,X\setminus T$ such that
$\ell_i$ is simple relative to $c_i$ for $i=1,2$, and that $c_1\cup
c_2$ is disjoint from $m_3$. (See Figure~\ref{fig:Y-null-leaf}.)  
Then there are three disjoint small Hopf links $h_1,h_2,h_3$ in
$N(T)\setminus(c_1\cup c_2)$ such that
\begin{enumerate}
\item if $T$ is of type I, there is a concordance between the pairs
$(W^T,c_1\cup c_2)$ and $(W^{h_1\cup h_2\cup h_3},c_1\cup c_2)$
that is strictly trivial on $(X\setminus \mathrm{Int}\,N(T))\times[a,b]$, and 
\item if $T$ is of type II, there is a concordance between the
  pairs
  \begin{equation*}
    ((S^{d-3}\times W)^T,S^{d-3}\times(c_1\cup c_2))
  \ \ \mbox{and} \ \ 
  (S^{d-3}\times W^{h_1\cup h_2\cup h_3},S^{d-3}\times(c_1\cup
    c_2))
\end{equation*}
  that is strictly trivial on the product
  $(S^{d-3}\times(X\setminus \mathrm{Int}\,N(T))\times[a,b]).$
\end{enumerate}
\end{Lem}
\begin{figure}[!htbp]
\[ \includegraphics[height=25mm]{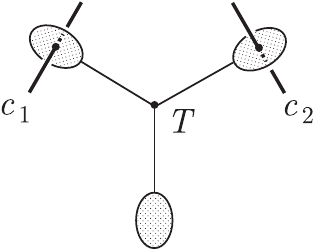} \]
\caption{Y-graph with null-leaf}\label{fig:Y-null-leaf}
\end{figure}
\begin{proof}
Case (1) is a corollary of Lemma~\ref{lem:Y-surgery-link-I}. It
suffices to delete $c_3$ and $c_3'$ in
Lemma~\ref{lem:Y-surgery-link-I}. By Property~\ref{propt:brunnian} of
Borromean rings, $(c_1'\cup c_2')\cap N(T)$ in
Lemma~\ref{lem:Y-surgery-link-I} is isotopic to $(c_1\cup c_2)\cap
N(T)$ fixing the boundary. Case (2) can be proven similarly by using
Lemma~\ref{lem:Y-surgery-link-II} instead of
Lemma~\ref{lem:Y-surgery-link-I}.
\end{proof}
\section{Family of framed links for graph surgery}
\subsection{Surgery on a collection of Y-graphs for a link}
Let $G$ be a connected uni-trivalent
graph embedded in $\mathrm{Int}\,X$ such that
\begin{enumerate}
\item $G$ has $r$ trivalent vertices and at least one univalent
  vertex,
\item edges are oriented in a way that the orientations of edges at
  each trivalent vertex is the same as that of Y-graph of Type I or
  II,
\item the univalent vertices of $G$ are on components of some
  spherical link $L$ in $\mathrm{Int}\,X$ consisting of 1- and
  $(d-2)$-spheres,
\item $L\cap \mathrm{Int}\,G=\emptyset$, where $\mathrm{Int}\,G$ is
  the complement of the union of univalent vertices in $G$,
\item each univalent vertex of $G$ that is ``inward" to $G$ is
  attached to a $(d-2)$-sphere in $L$,
\item each univalent vertex of $G$ that is ``outward" from $G$ is
  attached to a $1$-sphere in $L$.
\end{enumerate}
\par\vspace{5mm}
\begin{figure}[!htbp]
  \[ \includegraphics[height=30mm]{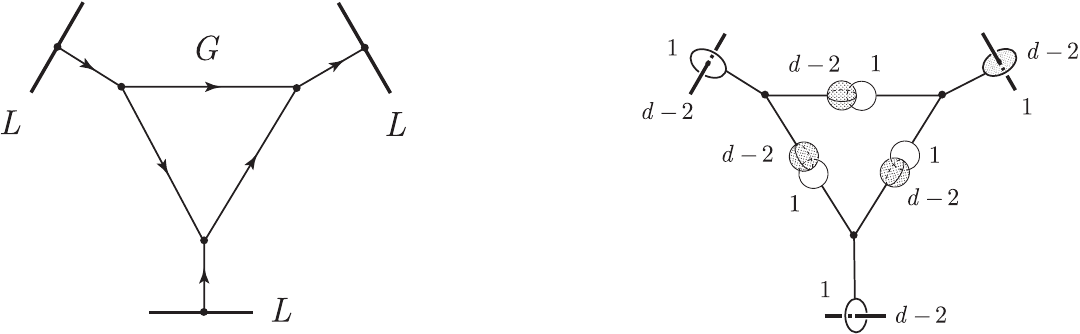} \]
\caption{A uni-trivalent graph $G$ attached to a link
  $L$ and the associated Y-link}\label{fig:G-NG-Y}
\end{figure}
We shall describe the effect of the surgery on $G$ in
Proposition~\ref{prop:family-emb}. To state
Proposition~\ref{prop:family-emb}, we introduce some notations.  We
take a small closed neighborhood $N(G)$ of $G$ such that its
intersection with $L$ consists of 1- and $(d-2)$-disks each of which
is a small neighborhood of a univalent vertex of $G$ in a component of
$L$. As before, we may construct a Y-link $G_1\cup\cdots\cup G_r$
inside $N(G)$ by putting a {framed} Hopf link at each edge of $G$
between trivalent vertices and by replacing each univalent vertex with
a leaf that bounds a disk in $N(G)$ transversally intersecting $L$ at
a point. We call such a leaf a {\it simple leaf} of $G$ relative to
$L$. Then we define surgery on $G$ by the surgery on the Y-link
$G_1\cup\cdots\cup G_r$.

Let $\mathbf{b}$ be a small $d$-disk and let $\mathbf{w}$
be the relative cobordism obtained from $\mathbf{b}\times I$ by
surgery along a small Hopf link $h$ in
$\mathrm{Int}\,\mathbf{b}\times\{1\}$. For a relative cobordism
$W$ between $\partial_-W={X}\times\{0\}$ and $\partial_+W\cong {X}$
such that $\partial W=\partial_+W\cup_{\partial
  {X}\times\{1\}}(\partial {X}\times I)\cup_{\partial
  {X}\times\{0\}}\partial_-W$, let $W_N$ denote the boundary connected
sum of $W$ and $N$ copies of $\mathbf{w}$ along disjoint union
of disks $D^{d-1}\times I\subset \partial {X}\times I$. (See
  Figure~\ref{fig:W_N}.)
\begin{figure}[!htbp]
\[ \includegraphics[height=40mm]{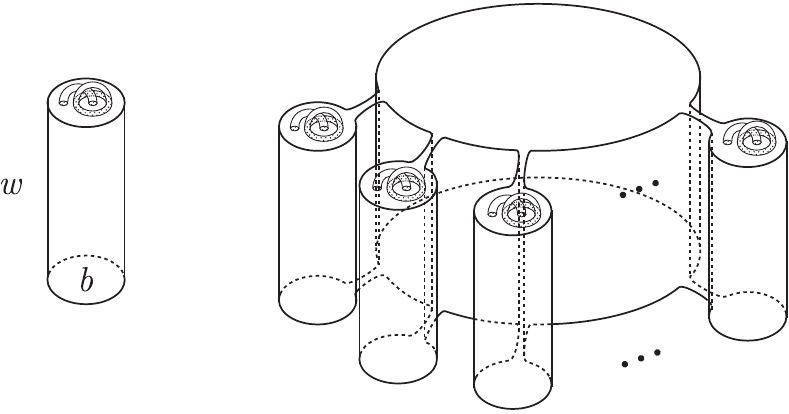} \]
\caption{The cobordism $W_N$}\label{fig:W_N}
\end{figure}

Let $B_G=S^{a_1}\times S^{a_2}\times\cdots\times S^{a_r}$, where
$a_i=0$ or $d-3$ depending on whether $G_i$ is of Type I or II,
respectively. Let $p^G:W^G\to B_G$ be the
$(W_{3r},\partial_\sqcup W_{3r})$-bundle obtained from the
trivialized $(X\times I)$-bundle over $B_G$ by surgery along the
associated set of families of {framed} Hopf links in {$X\times\{1\}$}
for $G$ as in Definitions~\ref{def:type-I} and \ref{def:type-II}.
\begin{Prop}\label{prop:family-emb}
Let $W=X\times[a,b]$, where we identify $X\times\{b\}$ with $X$.
Let $G$ be a connected uni-trivalent graph with $r$ trivalent
vertices, embedded in $\mathrm{Int}\,{X}$, and attached to some link
$L$ as above. Let $b_1,\ldots,b_{3r}$ be disjoint small $d$-balls in
$N(G)\setminus (L\cup G_1\cup\cdots\cup G_r)$ and let
$h_1,\ldots,h_{3r}$ be small Hopf links in $N(G)\setminus (L\cup
G_1\cup\cdots\cup G_r)$ such that $h_i\subset \mathrm{Int}\,b_i$. (See
Figure~\ref{fig:NG-b1}.) {Let $L_{N(G)}=L\cap N(G)$.} Then there is a
$B_G$-family of embeddings of the union of disks ${L_{N(G)}}$ into
$N(G)\setminus(b_1\cup\cdots\cup b_{3r})$:
\begin{equation*}
\Phi_s:{L_{N(G)}}\to N(G)\setminus(b_1\cup\cdots\cup b_{3r})\quad (s\in B_G)
\end{equation*}
that agree with the inclusion near $\partial N(G)$ such that there is
a concordance between the pairs
\begin{equation*}
(W^G,B_G\times L) \mbox{ and } 
  (B_G\times W^{h_1\cup\cdots\cup h_{3r}},\widetilde{L}),
\end{equation*}
where $\widetilde{L}=\bigcup_{s\in B_G} L_s$, which is a trivialized
subbundle of $B_G\times\partial_+W_{3r}=B_G\times{X}$, and $L_s$
is obtained from $L$ by replacing ${L_{N(G)}}$ by $\Phi_s$.  Moreover,
we may assume that the concordance is strictly trivial on
$B_G\times((X\setminus\mathrm{Int}\,N(G))\times[a,b])$, and for any
choice of a component $\ell$ of ${L_{N(G)}}$, we may assume that the
restriction of $\Phi_s$ to all the components in ${L_{N(G)}}\setminus
\ell$ does not depend on the parameter $s$, after a fiberwise isotopy,
which depends on the choice of $\ell$.
\end{Prop}
\begin{figure}[!htbp]
\[ \includegraphics[height=45mm]{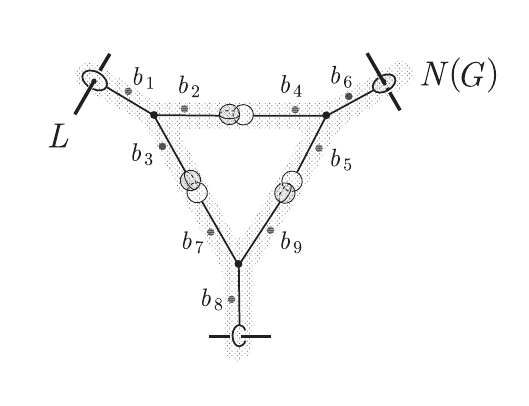} \]
\caption{The neighborhood $N(G)$ of the Y-link from
    Figure 15 is shaded, and each small ball $b_i$
    is embedded near the $i$-th edge of $G$}\label{fig:NG-b1}
\end{figure}

\begin{proof}
We prove this by induction on $r$. The case $r=1$ has been proved in
Lemma~\ref{lem:Y-surgery-link-I} or \ref{lem:Y-surgery-link-II}. 

For general $r$, we assume that the result holds true for connected
uni-trivalent graphs with at most $r-1$ trivalent vertices. Let $G$ be
a connected uni-trivalent graph with $r$ trivalent vertices as in the
statement. Then by assumption (condition (1) above) $G$ has a
trivalent vertex $v_r$ that is incident to a univalent vertex by a
single edge. We decompose $G$ into two parts by cutting the edges
incident to $v_r$ but not incident to univalent vertices, and apply
the induction hypothesis on the part with $r-1$ trivalent
vertices. More precisely, we put a framed Hopf link at each edge of
$G$ incident to $v_r$ but not incident to univalent vertices according
to the edge-orientation by the rule of
  Figure~\ref{fig:G-to-Y-link}, and then replace each univalent
vertex of $G$ with a simple leaf relative to $L$. We may assume that
this process yields a disjoint union of two connected objects: one is
a connected uni-trivalent graph with $r-1$ trivalent vertices with
some spheres attached to univalent vertices, and another is a
Y-graph. We denote the two components by $G'$ and $G''$,
respectively. We consider closed neighborhoods $N'$ and $N''$ of $G'$
and $G''$, respectively, given as follows. Let $\mu$ be the union of
spanning disks of the simple leaves of $G'$ each of which intersects
$L$ transversally by one point. Here, we assume that each leaf is a
round sphere about a point with small radius so that it has a
canonical flat spanning disk and we assume that $\mu$ is the union of
such flat disks. Now we take a small closed neighborhood $N'$ of
$G'\cup \mu$. Also, let $N''=N(G'')$. We assume that $N'$ is disjoint
from $G''$. On the other hand, $N''$ intersects $G'$ according to the
definition of $N(G'')$.

Let $\delta'$ be the union of $3(r-1)$ disjoint small Hopf links in $N'\setminus G'$. By induction hypothesis, we see that there are
$B_{G'}$-family of disks from $D'=L\cap N'$ inside $N'$:
\begin{equation*}
\varphi_{s'}:D'\to N'\quad (s'\in B_{G'}),
\end{equation*}
and a concordance between
\begin{equation*}
\begin{split}
  &(W^{G'},B_{G'}\times L)
  \mbox{ and }
  (B_{G'}\times W^{\delta'},
  \widetilde{L}'),
\end{split}
\end{equation*} 
where $\widetilde{L}'=\bigcup_{s'\in B_{G'}}L_{s'}$ and $L_{s'}$ is
obtained from $L$ by replacing {$D'$} by $\varphi_{s'}$. Moreover,
we may assume that the restriction of $\varphi_{s'}$ to the
components of $D'$ except one, is a strictly trivial family.

We next consider the effect of the surgery on $G''$.  Let $\lambda$ be
the disjoint union of all the leaves of $G'$ that intersect $N''$, and
let $L''=L\cup \lambda$.  Let $\delta''$ be the union of three
disjoint small Hopf links in $N''\setminus G''$. By
Lemma~\ref{lem:Y-surgery-link-I} or \ref{lem:Y-surgery-link-II}, we
see that there are $S^{a_r}$-families of leaves
$\widetilde{\lambda}=\bigcup_{s''\in S^{a_r}}\lambda_{s''}$ in $N'\cup
N''$ and a concordance between
\begin{equation*}
(W^{G''},S^{a_r}\times L'') \ \ \ \mbox{and} \ \ \ 
(S^{a_r}\times W^{\delta''},\widetilde{L}''),
\end{equation*}
  where $\widetilde{L}''=\widetilde{\lambda}\cup (S^{a_r}\times L)$,
  such that for each $s''$, $\lambda_{s''}$ agrees with $\lambda$ near
  the basepoint of the leaf. The replacement of $\lambda$ with
  $\lambda_{s''}$ gives a family of embeddings of $G'$ in
  $N(G)\setminus L$:
\begin{equation*}
  \varphi_{s''}:G'\to N(G)\setminus L\quad (s''\in S^{a_r}).
\end{equation*}
By isotopy extension, the family $\varphi_{s''}$ can be extended to a
family of embeddings $\varphi_{s''}':N'\to N(G)$ ($s''\in
S^{a_r}$).

Now we combine the two surgeries for $G'$ and $G''$. Let
$\widetilde{L}$ be the trivialized subbundle of $B_G\times {X}$
obtained from $B_G\times L$ by replacing $B_G\times L_{N'}$
 by the composition
\begin{equation*}
\varphi_{s''}'\circ \varphi_{s'}:D'\to
N(G)\quad (s',s'')\in B_G.
\end{equation*}
By the results of the previous paragraphs, we see that there is a
concordance between the pairs $(W^G,B_G\times L)$ and
$(B_G\times W^{\delta'\cup \delta''},\widetilde{L})$. Note that
the restriction of $\widetilde{L}$ to the components of $L_{N(G)}$
that intersect $N'$ is a strictly trivial family except one
component, and also the restriction of $\widetilde{L}$ to the
  components of $L_{N(G)}$ that intersects $N''$ is a strictly trivial
  family.  This completes the induction.
\end{proof}

\subsection{Replacing a trivalent graph with a family of Hopf links}
Now we shall nearly complete the proof of
Theorem~\ref{thm:bordism-simplify} (1), by proving the corresponding
statement for $B_\Gamma$-family instead of $S^{k(d-3)}$-family:

\begin{Prop}\label{prop:Hopf-Gamma}
Let $\Gamma$ be a labeled edge-oriented trivalent graph as in
section~\ref{s:graph-surgery} {with $2k$ vertices}. The
{$(X,\partial)$}-bundle $\pi^\Gamma:E^\Gamma\to B_\Gamma$ for an
embedding $\phi:\Gamma\to \mathrm{Int}\,{X}$ is concordant to a
{$(X,\partial)$}-bundle obtained from the product bundle $
S^{k(d-3)}\to S^{k(d-3)}\times {X}$ by fiberwise {surgeries} along a
$B_\Gamma$-family of framed links $h_s:S^1\cup S^{d-2}\to
\mathrm{Int}\,{X}$, $s\in B_\Gamma$, that satisfies the following
conditions:
\begin{enumerate}
\item $h_s$ is isotopic to the Hopf link for each $s$.
\item The restriction of $h_s$ to $S^{d-2}$ component is a constant
  $B_\Gamma$-family.
\item There is a small neighborhood $N$ of $\mathrm{Im}\,\phi$ such
  that the image of $h_s$ is included in $N$ for all $s\in B_\Gamma$.
\end{enumerate}
\end{Prop}
We will prove this by trying to construct a concordance between the
families of cobordisms for the two surgeries and by restriction to the
top faces. Of course, there is no such concordance in the obvious
sense since the numbers of components of the framed links for
$\Gamma$-surgery and surgery along a family of Hopf-links are
different. We modify the assumption slightly so that a concordance
between the two families of cobordisms will make sense.

Now we set $W={X}\times I$ and let $p^\Gamma:W^\Gamma\to B_\Gamma$
be the $(W_{6k},\partial_\sqcup W_{6k})$-bundle obtained from the
trivial $W$-bundle by surgery along the associated family of
($6k\times 2=12k$ component) framed links in ${X}\times\{1\}$ for the
$\Gamma$-surgery. The restriction of this bundle to the top face gives
the former {$(X,\partial)$}-bundle $\pi^\Gamma:E^\Gamma\to B_\Gamma$
of Proposition~\ref{prop:Hopf-Gamma}.  The number $6k$ is because
there are $2k$ Y-graphs for the $\Gamma$-surgery each gives rise to 3
Hopf links.  On the other hand, the latter {$(X,\partial)$}-bundle of
Proposition~\ref{prop:Hopf-Gamma} is the top face of a
$(W_1,\partial_\sqcup W_1)$-bundle over $B_\Gamma$.

We add to $W^\Gamma$ one more Hopf link surgery without changing the
{$(X,\partial)$}-bundle on the top face, as follows. Let $G_1\cup
\cdots\cup G_{2k}$ be the Y-link for the embedding $\phi$ of
$\Gamma$. Let $a_1\cup b_1$ be the framed Hopf link for the first edge
of $\Gamma$ as in Figure~\ref{fig:G-to-Y-link}, which are leaves of
some Y-graphs. We replace $a_1\cup b_1$ by a framed ``Hopf chain"
$c_1\cup c_2\cup c_3\cup c_4$ such that
\begin{itemize}
\item $\dim{c_1}=\dim{c_3}=\dim{a_1}$, $\dim{c_2}=\dim{c_4}=\dim{b_1}$,
\item $c_i\cup c_{i+1}$ is a Hopf link for $i=1,2,3$.
\end{itemize} 
\begin{figure}[!htbp]
\[ \includegraphics[height=15mm]{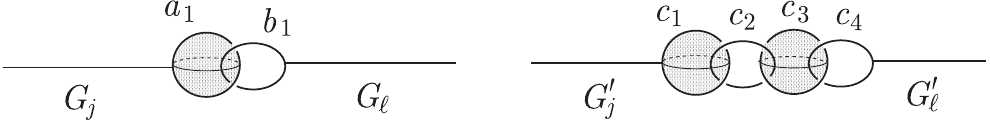} \]
\caption{The Hopf chain $c_1\cup c_2\cup c_3\cup
  c_4$}\label{fig:hopf-ch}
\end{figure}
Then the leaves $a_1$ and $b_1$ are replaced by $c_1$ and $c_4$,
respectively, and $G_1\cup \cdots\cup G_{2k}$ becomes a Y-link
$G_1'\cup \cdots \cup G_{2k}'$ that is linked to the Hopf link
$c_2\cup c_3$. The Y-link $G_1'\cup\cdots\cup G_{2k}'$ is the one
obtained from a uni-trivalent graph $G$ attached to the link
$L=c_2\cup c_3$, as in Proposition~\ref{prop:family-emb} below. By
Lemma~\ref{lem:Hopf-surgery}, this replacement does not change the
concordance class of the pair up to small Hopf links. Namely, there
are a small Hopf link $h$ in $\partial_+W$ that is disjoint from the
Y-link $G_1\cup\cdots\cup G_{2k}$ and $L$, and a family of concordances between
the pairs
\begin{equation*}
  (W^\Gamma,B_\Gamma\times h) \ \ \ \mbox{ and }
  \ \ \ (W^{G_1'\cup\cdots\cup G_{2k}'}, B_\Gamma\times(c_2\cup
    c_3))
\end{equation*}
parameterized by $B_\Gamma$.  Let
$p_1^\Gamma:W_1^\Gamma\to B_\Gamma$ be the $(W_{6k+1},\partial_\sqcup
W_{6k+1})$-bundle given by fiberwise surgery
\begin{equation*}
  (W^{G_1'\cup\cdots\cup G_{2k}'})^{B_\Gamma\times(c_2\cup c_3)}.
\end{equation*}
The number $6k+1$ is due to the addition of $c_2\cup c_3$. The newly
added Hopf link $c_2\cup c_3$ will serve as the family of Hopf links
$h_s$ of Proposition~\ref{prop:Hopf-Gamma}.

Proposition~\ref{prop:Hopf-Gamma} is an immediate corollary of the
following lemma, which gives an extension of
Proposition~\ref{prop:Hopf-Gamma} to cobordisms.
\begin{Lem}\label{lem:Gamma-concordance}
Let $\Gamma$ be as in Proposition~\ref{prop:Hopf-Gamma}.  The above
$(W_{6k+1},\partial_\sqcup W_{6k+1})$-bundle $p^\Gamma_1:W^\Gamma_1\to
B_\Gamma$ determined by an embedding $\phi:\Gamma\to
\mathrm{Int}\,{X}\times\{1\}$ is concordant to a
$(W_{6k+1},\partial_\sqcup W_{6k+1})$-bundle that is obtained from the product
$W$-bundle $B_\Gamma\times W\to B_\Gamma$ by fiberwise handle
attachments along some $B_\Gamma$-family of framed links $h_s:S^1\cup
S^{d-2}\to \mathrm{Int}\,{X}\times\{1\}$, $s\in B_\Gamma$, and
fiberwise boundary connected sums with $6k$ copies of the trivial
$(\mathbf{w},\partial_\sqcup \mathbf{w})$-bundle
$p_0:B_\Gamma\times \mathbf{w}\to B_\Gamma$,
where $h_s$ satisfies the conditions (1), (2), (3) of
Proposition~\ref{prop:Hopf-Gamma}.
\end{Lem}
\begin{proof}
We assume without loss of generality that $\dim{c_2}=1$ and
$\dim{c_3}=d-2$. Applying Proposition~\ref{prop:family-emb} for the
Y-link $G_1'\cup\cdots\cup G_{2k}'$ and $L=c_2\cup c_3$, we see that
surgery on $G_1'\cup \cdots\cup G_{2k}'$ produces a $B_\Gamma$-family
of embeddings of $L\cap N(G)$ into $N(G)$, whose restriction to
$c_3\cap N(G)$ is a trivial family. This gives the desired family of
framed Hopf links.
\end{proof}

\section{Bordism modification to a $S^{k(d-3)}$-family
  of surgeries}\label{s:bordism}
\subsection{From a $B_\Gamma$-family to a $S^{k(d-3)}$-family}
We shall complete the proof of Theorem~\ref{thm:bordism-simplify} (1).
\begin{Prop}\label{prop:family-emb-bordism}
Let $G$ be a uni-trivalent graph attached to a framed link $L$, as in
Proposition~\ref{prop:family-emb}. The $B_G=S^{a_1}\times\cdots\times
S^{a_r}$-family of framed embeddings of disks {$L_{N(G)}=L\cap N(G)$}
in $N(G)$ of Proposition~\ref{prop:family-emb} can be deformed into an
$S^{a_1+\cdots+a_r}$-family by an oriented bordism in the space
$\fEmb_\partial({L_{N(G)}},N(G))$.
\end{Prop}
To prove Proposition~\ref{prop:family-emb-bordism}, we shall instead
prove the following stronger lemma.
\begin{Lem}\label{lem:family-emb-bordism}
The map $B_G\to \fEmb_\partial({L_{N(G)}},N(G))$ for the $B_G$-family
of Proposition~\ref{prop:family-emb} factors up to homotopy over a map
$B_G\to S^{a_1+\cdots+a_r}$ of degree 1.
\end{Lem}
\begin{proof}
We prove this by induction on $r$. The case $r=1$ is obvious. Assume
that the map
\begin{equation*}
g_{r-1}:B_{G'}=S^{a_1}\times\cdots\times S^{a_{r-1}}\to
\fEmb_\partial({L_{N(G)}},N(G'))
\end{equation*}
for a Y-link $G_1\cup\cdots\cup G_{r-1}$ that corresponds to a
connected uni-trivalent graph $G'$ factors up to homotopy into a
degree 1 map $S^{a_1}\times\cdots\times S^{a_{r-1}}\to
S^{a_1+\cdots+a_{r-1}}$ and a map
\begin{equation*}
  \overline{g}_{r-1}:S^{a_1+\cdots+a_{r-1}}\to \fEmb_\partial({L_{N(G')}},N(G')).
\end{equation*}
Since $g_{r-1}$ is null-homotopic, one may apply Lemma~\ref{lem:BY}
below, and the map $\overline{g}_{r-1}$ is null-homotopic.

Adding one more Y-graph $G_r$ so that $G_1\cup\cdots\cup G_r$
corresponds to a connected uni-trivalent graph $G$, we obtain a map
$g_r:B_{G'}\times S^{a_r}\to \fEmb_\partial({L_{N(G)}},N(G))$ that
factors up to homotopy over a degree 1 map $B_{G'}\times S^{a_r}\to
S^{a_1+\cdots+a_{r-1}}\times S^{a_r}$.

The restrictions of the induced map
\begin{equation*}
  \overline{g}_r:S^{a_1+\cdots+a_{r-1}}\times S^{a_r}\to
  \fEmb_\partial({L_{N(G)}},N(G))
\end{equation*}
to the subspaces $S^{a_1+\cdots+a_{r-1}}\times\{*\}$ and $\{*\}\times
S^{a_r}$ of $S^{a_1+\cdots+a_{r-1}}\times S^{a_r}$ are pointed
null-homotopic in $\fEmb_\partial({L_{N(G)}},N(G))$ by
Lemma~\ref{lem:null-leaf-I} and by the
nullity of $\overline{g}_{r-1}$ in
\begin{equation*}
  \fEmb_\partial({L_{N(G')}},N(G')).
\end{equation*}
Thus the map $\overline{g}_r$ factors up to homotopy over a degree 1
map $S^{a_1+\cdots+a_{r-1}}\times S^{a_r}\to S^{a_1+\cdots+a_r}$.
\end{proof}
\begin{Lem}[{\cite[Proof of Lemma~B]{Wa18a}}]\label{lem:BY}
Let $B=S^{a_1}\times S^{a_2}\times\cdots\times S^{a_s}$ and let
\begin{equation*}
\begin{split}
  A=(\{*\}\times S^{a_2}\times\cdots\times S^{a_s})
  \cup (S^{a_1}\times\{*\}\times\cdots\times S^{a_s})
  \cup\cdots\cup (S^{a_1}\times S^{a_2}\times\cdots\times\{*\}),
\end{split}
\end{equation*}
For a space $Y$, suppose that we have a pointed
null-homotopy of a pointed map $g: B\to Y$ and another pointed
null-homotopy of the restriction $g|_A:A\to Y$. Then $g$ can be
factored up to homotopy into a pointed map $B\to B/A\simeq
S^{a_1+\cdots a_s}$ and a null-homotopic map $B/A\to Y$.
\end{Lem}
\begin{Cor}\label{cor:bordism}
The $B_\Gamma$-family of framed links $h_s:S^1\cup S^{d-2}\to
\mathrm{Int}\,{X}$, $s\in B_\Gamma$, in
Proposition~\ref{prop:Hopf-Gamma} can be deformed by a bordism in the
space of embeddins, into a $S^{k(d-3)}$-family of framed embeddings $S^1\cup S^{d-2}\to \mathrm{Int}\,X$ that satisfies the following conditions:
\begin{enumerate}
\item $h_s$ is isotopic to the Hopf link for each $s$.
\item The restriction of $h_s$ to $S^{d-2}$ component is a constant
  $S^{k(d-3)}$-family.
\item There is a small neighborhood $N$ of $\mathrm{Im}\,\phi$ such
  that the image of $h_s$ is included in $N$ for all $s\in
  S^{k(d-3)}$.
\end{enumerate}
Hence, fiberwise handle attachments along the family of embeddings
over the bordism gives a bundle bordism of cobordism bundles
$p_1^\Gamma:W_1^\Gamma\to B_\Gamma$ to a
$(W_{6k+1},\partial_\sqcup W_{6k+1})$-bundle over $S^{k(d-3)}$,
which restricts on the top face to a $({X},\partial)$-bundle bordism
between $\pi^\Gamma:E^\Gamma\to B_\Gamma$ and a
$({X},\partial)$-bundle $\varpi^\Gamma:\overline{E}^\Gamma\to
S^{k(d-3)}$.
\end{Cor}
\subsection{Modification into a family of $h$-cobordisms}
We prove Theorem~\ref{thm:bordism-simplify} (2).
\begin{Prop}\label{prop:pseudoisotopy}
There exists a $({X}\times I,\partial_\sqcup ({X}\times I))$bundle $\Pi^\Gamma: W^\Gamma_h\to S^{k(d-3)}$ such that
\begin{enumerate}
\item the fiberwise restriction of $\Pi^\Gamma$ to ${X}\times \{1\}$
  is $\varpi^\Gamma$,
\item $W^\Gamma_h$ is obtained by attaching $S^{k(d-3)}$-families of
  1- and 2-handles to the product ${X}\times I$-bundle
  $S^{k(d-3)}\times ({X}\times I)\to S^{k(d-3)}$ at
  $S^{k(d-3)}\times ({X}\times\{1\}).$.
\end{enumerate}
\end{Prop}
\begin{proof}
By Corollary~\ref{cor:bordism}, there is a cobordism bundle $E\to
S^{k(d-3)}$ which is obtained from the trivialized $({X}\times
I)$-bundle by attaching families of 2- and $(d-1)$-handles along $h_s$
and whose restriction to $\partial_+E$ agrees with $\varpi^\Gamma$.

Let $E_1$ be the family of handlebodies obtained by attaching only the
family of $(d-1)$-handles to the strictly trivial $({X}\times
I)$-bundle by $h_s|_{S^{d-2}}$.  Since the attaching map
$h_s|_{S^{d-2}}$ of the family of $(d-1)$-handles is a strictly
trivial family by Corollary~\ref{cor:bordism}, the family $E_1$ is a
strictly trivial bundle, on the top of which the 2-handle may be
attached along the attaching sphere induced by $h_s|_{S^1}$ on
$\partial_+E_1$ that may not be strictly trivial.

Attaching a $(d-1)$-handle to ${X}\times I$ along an unknotted framed
$(d-2)$-sphere on ${X}\times\{1\}$ turns the top face into ${X}\#
(S^{d-1}\times S^1)$. Also, the same manifold can be obtained by
attaching a 1-handle along a framed 0-sphere on ${X}\times \{1\}$
instead of a $(d-1)$-handle. Thus, we may replace the strictly trivial
bundle $E_1$ by another family $E_1'$ of handlebodies that is obtained
by attaching strictly trivial family of 1-handles to ${X}\times I$,
without changing the manifold
\begin{equation*}
  \partial_+ E_1= S^{k(d-3)}\times({X}\#(S^{d-1}\times S^1)).
\end{equation*}
  Then we attach a family of 2-handles to $E_1'$ along the attaching
  spheres induced by $h_s|_{S^1}$ on $\partial_+E_1'=\partial_+E_1$.
\begin{figure}[!htbp]
\[\includegraphics[height=33mm]{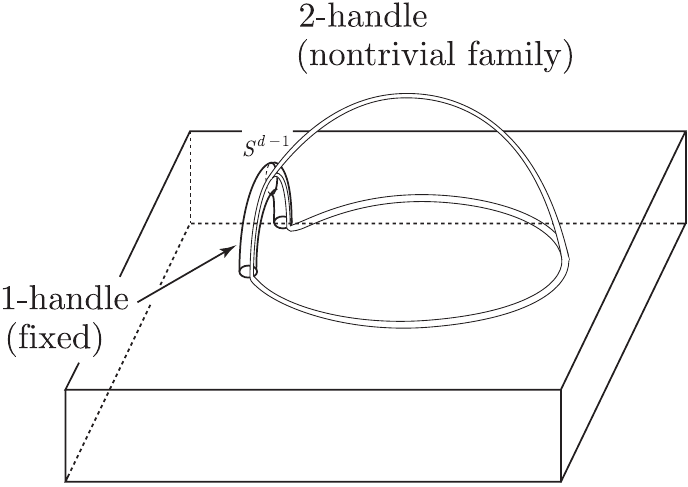} \]
\caption{}\label{canceling-handles}
\end{figure}
The resulting bundle $\Pi^\Gamma:W_h^\Gamma\to S^{k(d-3)}$ is a
$({X}\times I)$-bundle, since the two handles are in a cancelling
position in a fiber, namely, the descending disk of the 2-handle and
the ascending disk of the 1-handle intersects transversally in one
point in $\partial_+E_1$. Then by M.~Morse's result \cite{Mo} (see
also \cite[Theorem~5.4 (First Cancellation Theorem)]{Mi}), the pair of
two handles can be eliminated and the cobordism can be modified into
the trivial $h$-cobordism.  By construction,
$\partial_+W_h^\Gamma=\overline{E}^\Gamma$ and $\pi^\Gamma$ restricts
to $\varpi^\Gamma$.
\end{proof}
\begin{Rem}\label{rem-morse}
We notice that by construction, the bundle $\Pi^\Gamma: W_h^\Gamma\to
S^{k(d-3)}$ admits a fiberwise Morse function $f:W_h^\Gamma\to \R$ and
a fiberwise gradient-like vector field $\xi$ for $f$ such that the
family of handle decompositions for $\xi$ agrees with that of the 1-
and 2-handles in Proposition~\ref{prop:pseudoisotopy}. Such a family
of Morse functions can be constructed by applying
\cite[Theorem~3.12]{Mi} for the families of handles, which is possible
since the families of handles are given by families of attaching maps,
and the construction of the Morse function in the proof of
\cite[Theorem~3.12]{Mi} for the surgery $\chi(V,\varphi)$ depends
smoothly on the attaching maps $\varphi$.
\end{Rem}
\begin{Rem}
A similar trick to turn the constant $(d-1)$-handle upside down into a 1-handle was recently used in a similar setting by David Gay in \cite{Ga}.
\end{Rem}


\section{Proof of Theorems \ref{thm:M} and \ref{thm:C}}
As noted in the introduction, Theorem~\ref{thm:M} is an immediate
corollary of Theorem~ \ref{thm:C}, which we now prove.  We will use a
result from \cite{BHSW} and the following definition.
\begin{Def}\label{folds}(cf. \cite[Definition 2.7]{BHSW})
Let $\pi: E\to B$ be a smooth bundle of $(d+1)$-dimensional
cobordisms. Denote by $V_b$ a fiber over a point $b\in B$, where
$V:=V_b$ is a relative cobordism between $\p_0 V$ and $\p_1 V$ such
that $\p V = \p_0 V\cup \bar\p V \cup \p_1V$.  We assume that a
structure group of the bundle $\pi: E\to B$ is $\Diff_{\sqcup}(V)$,
i.e. of those diffeomorphisms which restrict to the identity near
$\p_0 V\cup \bar\p V$.  We denote by $E_0$, $ E_{\bar\p}$ and $E_1$ a
restriction of the bundle $E$ to the fibers $\p_0 V$, $\bar\p V$ and
$\p_1V$ respectively.  A smooth map $F: E \to B\times I$ is said to be
an admissible family of Morse functions or \emph{admissible with fold
singularities with respect to $\pi$} if it satisfies the following
conditions:
\begin{enumerate}
\item
 The diagram
\begin{equation*}
\begin{diagram}
\setlength{\dgARROWLENGTH}{1.6em}
 \node{}
\node[2]{E}
      \arrow{s,l}{\pi}
      \arrow[2]{e,t}{F}
\node[2]{B\times I}
      \arrow{wsw,b}{p_1}
\\
\node[3]{B}
\end{diagram}
\end{equation*}
commutes.  Here $p_1: B\times I\to B$ is projection on the first
factor.
\item
The pre-images $F^{-1}(B\times\{0\})$ and  
$F^{-1}(B\times\{1\})$ coincide with the submanifolds $E_0$ and $E_1$ 
respectively. 
\item
  The set $\Cr(F)\subset E$ of critical points of $F$ is contained in
  $E\setminus (E_0\cup E_{\p}\cup E_1)$ and near each critical point
  of $F$ the bundle $\pi$ is equivalent to the trivial bundle $\R^k
  \times \R^{d+1} \stackrel{p_1}{\to} \R^k$ so that with respect to
  these coordinates on $E$ and on $B$ the map $F$ is a standard map
  $\R^{k} \times \R^{d+1} \to \R^k \times \R$ with a fold singularity.
\item
For each $z\in B$ the restriction
\begin{equation*}
f_b=F|_{V_b}: V_b\to \{b\}\times I \stackrel{p_2}{\longrightarrow} I
\end{equation*}
is an admissible Morse function, i.e. its critical points have indices
$\leq d-2$.
\end{enumerate}  
\end{Def}
\begin{Lem}\label{lemma-admis}
Let $\Pi^\Gamma: W_h^\Gamma\to S^{k(d-3)}$ be as in
  Proposition~\ref{prop:pseudoisotopy}.  Then each $({X}\times
I,\partial_\sqcup({X}\times I))$-bundle $\Pi^\Gamma: W_h^\Gamma\to
S^{k(d-3)}$ of $h$-cobordisms has an admissible family of Morse
functions as above provided $d\geq 4$.
\end{Lem}  
\begin{proof}
  Let $\Pi^\Gamma: W_h^\Gamma\to S^{k(d-3)}$ be a
    $({X}\times I,\partial_\sqcup({X}\times I))$-bundle of
    $h$-cobordisms as above.  We just noticed in Remark
    \ref{rem-morse} that there is a fiberwise Morse function $f:
    W_h^\Gamma\to \R$ which agrees with that of the 1- and 2-handles
    (in the case when $d$ is even). In the case when $d$ is odd we
    deal with $m$- and $(m+1)$-handles, where $m=(d-1)/2$.  In
    particular, the codimension of the handles (inside the
    $h$-cobordisms) is at least three for all $d\geq 4$.
\end{proof}
\begin{proof}[Proof of Theorem~\ref{thm:C}] 
%
We consider bundles of $h$-cobordisms we have constructed. In both
cases, when $d$ is even or odd, Lemma \ref{lemma-admis} guarantees
that such a bundle satisfies the conditions of \cite[Theorem
  2.9]{BHSW} since $\Pi^{\Gamma}: W_h^\Gamma\to
  S^{k(d-3)}$ has only admissible fold singularities. Thus we obtain
that every fiber has a psc-metric. This proves Theorems \ref{thm:M}
and \ref{thm:C}.
\end{proof}
\begin{Rem}
We should admit that \cite[Theorem 2.9]{BHSW} assumes that the
structure group is $\Diff_{\p}(V)$; however, it is easy
to see that the same proof works in more general situation, in
particular when the structure group is $\Diff_{\sqcup}(V)$.
\end{Rem}  


\end{document}